\documentclass[preprint,12pt,numbers,sort&compress]{elsarticle}

\usepackage{enumerate,amsfonts,dsfont,amsmath,accents,amsthm,mathrsfs,cases,bbm,amssymb}
\usepackage{setspace}
\usepackage{graphicx,float}

\usepackage{hyperref}
\allowdisplaybreaks      
\newtheorem{theorem}{Theorem}[section]
\newtheorem{lemma}[theorem]{Lemma}
\newtheorem{proposition}[theorem]{Proposition}

\theoremstyle{definition}

\newtheorem{remark}[theorem]{Remark}
\newtheorem{claim}{\bf{Claim}}

\numberwithin{equation}{section}

\journal{Journal of \LaTeX\ Templates}

\begin{document}
\begin{frontmatter}
\title{{\bf On the $p$-Laplacian Lichnerowicz equation on compact Riemannian manifolds}}

\author[mymainaddress]{Nanbo Chen}
\ead{flyingnb@126.com}

\author[mymainaddress,mysecondaryaddress]{Xiaochun Liu\corref{mycorrespondingauthor}}
\cortext[mycorrespondingauthor]{Corresponding author}
\ead{xcliu@whu.edu.cn}

\address[mymainaddress]{School of Mathematics and Statistics, Wuhan University,
Wuhan 430072, China}
\address[mysecondaryaddress]{Hubei Key Laboratory of Computational Sciences, Wuhan University, Wuhan 430072, China}

\begin{abstract}
In this paper, we deal with a singular quasilinear critical elliptic equation of Lichnerowicz type involving the $p$-Laplacian operator.
With the help of the subcritical approach from variational method,
we obtain the non-existence, existence, and multiplicity results under some given assumptions.
\end{abstract}

\begin{keyword}
$p$-Laplacian\sep  Critical exponent\sep Negative exponent\sep Variational methods\sep Compact Riemannian manifold.
\MSC[2010] 58J05\sep 35J20.
\end{keyword}

\end{frontmatter}


\section{Introduction and main results}

Let $(M,g)$ be a smooth compact Riemannian manifold of dimension $n\geq 3$ without boundary.
We investigate the following $p$-Laplacian Lichnerowicz type equation in $M$:
\begin{equation}\label{eq:M}
\left\{
\begin{aligned}
&\Delta_{p,g}u+h u^{p-1}=f(x)u^{p^{*}-1}+{a(x)}{u^{-p^*-1}}, \\
&u>0,
\end{aligned}
\right.
\end{equation}
where $1<p<n$, $f(x)$ and $a(x)$ are smooth functions on $M$, and $h$ is a \textbf{negative} constant.
Here $p^{*}=\frac{np}{n-p}$ is the critical Sobolev exponent for the embedding of $H_1^p(M)$ into Lebesgue spaces, and
$\Delta_{p,g}:= -\operatorname{div}_g(|\nabla_g u|_g^{p-2}\nabla_g u)$ is the \textit{$p$-Laplace-Beltrami operator} associated to the metric $g$ on $M$,
which is defined in local
coordinates by the expression
$$\Delta_{p,g}u =-\frac{1}{\sqrt{|g|}}\frac{\partial}{\partial x^i}(\sqrt{|g|}g^{ij}|\nabla_g u|_g^{p-2}\frac{\partial u}{\partial x^j}),$$
where $(g^{ij})$ is the inverse of the metric matrix $(g_{ij})$ and ${|g|}:={\textrm{det}(g_{ij})}$ is the determinant of the metric tensor.

Such type of equations arises from the Hamiltonian constraint equation for the Einstein-scalar field system in general relativity.
See for example, \cite{CIP1, CIP2, Ngo2016} and references therein.
In the semilinear case $p = 2$, with the help of the conformal method, one is led to a simple scalar equation,
which is named as the Einstein-scalar field Lichnerowicz equation (the Lichnerowicz equation, in short).
Such equations have been the subject of extensive study in recent years due to the nature of their origin.

In the case $p = 2$, there are interesting papers written by Ng\^{o} and Xu (see \cite{Ngo1}, \cite{Ngo2} and \cite{Ngo3}). In particular,
in \cite{Ngo1}, they obtained the non-existence, existence, and multiplicity results
for positive solutions for the following Lichnerowicz equation
\begin{equation}\label{e1.1}
\Delta_{g}u+hu=f(x)u^{2^{*}-1}+\frac{a(x)}{u^{2^*+1}},\quad u>0,\quad \text{in }M,
\end{equation}
where $h<0$ is a constant, $f(x)$ and $a(x)\geq0$ are smooth functions. We should mention that the method they used is based on the method of Rauzy (\cite{Rauzy})
for the prescribed scalar curvature problem on a compact Riemannian manifold with negative conformal invariant.
In \cite{Ma13}, Ma and Wei studied the stability and multiplicity to Lichnerowicz equation \eqref{e1.1} for  $h > 0$, $f(x) > 0$ and $a(x)\geq0$.
In \cite{Hebey3}, Hebey-Pacard-Pollack established some non-existence and existence results for positive solutions of \eqref{e1.1} for the case $h>0$.
In the papers \cite{CIP2} and \cite{Isen} of Choquet-Bruhat-Isenberg-Pollack and Isenberg, more advanced existence results were obtained via the sub- and supersolution method for elliptic equations. Some further interesting results on the Lichnerowicz equation have been obtained
in \cite{CIP1}, \cite{DrHe}, \cite{MaXu}, \cite{MaLi}, \cite{MaSun},  \cite{Prem15}, \cite{Zhao0} and \cite{Zhao1}.

For the case $p > 1$, the $p$-Laplacian Lichnerowicz equation is a special case of the following so-called
\textit{generalized scalar curvature type equation}:
\begin{equation}
\Delta_{p,g}u+h(x)u^{p-1}=f(x)u^{p^{*}-1}+g(x,u), \quad  u>0, \quad \text{in }M. \label{e1.2}
\end{equation}
Such equation is nonlinear, of degenerated elliptic type, and of critical Sobolev growth, which arises quite naturally in
many branches of mathematics. For instance, in differential geometry it is an extension of the equation of prescribed scalar curvature.
The existence of solutions for problem \eqref{e1.2} has been extensively studied, and
many fruitful results have been obtained. For example, when $g(x,u)=0$ ( i.e. $a(x)=0$ in \eqref{eq:M}),
Druet (\cite{Druet}) studied the existence of positive solutions
under suitable assumptions on manifold $M$, $h(x)$ and $f(x)$.
Later, Benalili and Maliki (\cite{Be06}) extended the corresponding results to the complete Riemannian manifolds.
We refer to Hebey (\cite{Hebey1}) for more details and nice applications.
Recently, Chen and Liu (\cite{ChenLiu}) investigated \eqref{e1.2} and obtain some existence results when $g(x,u)$ is a lower order perturbation in the sense that
$$\lim\limits_{|u|\to\infty} \frac{g(x,u)}{|u|^{p^*-1}}=0,\quad \text{uniformly for } x\in M.$$
More results on the equation \eqref{e1.2} on a Riemannian manifold $(M,g)$
can be found in \cite{Be09, Be10, Silva1, Zhao2, Zhao3} and the reference therein.

Motivated by the above-mentioned work, in this paper, we will establish the non-existence and existence theorems for \eqref{eq:M}.
The major difficulties come from the following three aspects: \textit{critical Sobolev embedding exponent}, \textit{a sign-changing potential $f(x)$}, and \textit{a negative power nonlinearity}.
To overcome these difficulties, we generalize the analysis techniques developed in \cite{Ngo1} for $p = 2$ (see also \cite{Rauzy}).
The approach is variational and based on the so-called subcritical approach which is widely known for solving the Yamabe problem (see, e.g., \cite{Aubin1, MS}).

Now, our first existence result reads as follows, which the function $f$ involved in the nonlinearity is of changing sign.
\begin{theorem} \label{theorem1}
Let $(M, g)$ be a smooth compact Riemannian manifold without boundary of dimension $n$ $(n \geq 3)$. Let $h<0$ be a constant,
$f$ and $a\geq 0$ be smooth functions on $M$ with $\int_Ma\,dv_g>0$, $\int_Mf\,dv_g<0$, $\sup_Mf>0$ and $|h|<\lambda_f$ where
$\lambda_f$ is given in \eqref{eq2.2}.
Moreover, suppose that the integral of $a$ satisfies
\begin{equation} \label{e1.3}
\int_{M}a\,dv_g< \frac{p}{2(n-p)}\Big(\frac{2n-p}{2(n-p)}\Big)^{\frac{2n}{p}-1}
\Big(\frac{|h|}{\int_{M}|f^-|dv_g}\Big)^{\frac{2n}{p}}\int_M|f^-|\,dv_g,
\end{equation}
where $f^-$ is the negative part of $f$. Then there exists a number $\mathcal{C}>0$ such that, if
\begin{equation} \label{e1.4}
{\sup_Mf}{\left(\int_{M}|f^-|\,dv_g\right)^{-1}}\leq \mathcal{C},
\end{equation}
problem \eqref{eq:M} admits at least two positive $C^{1,\alpha}(M)$ solutions with $\alpha \in (0, 1)$.
\end{theorem}

As a remark, by straightforward calculus, a necessary condition for Eq. \eqref{eq:M} to admit a positive solution is $\int_{M}f\,dv_g<0$.
As another remark, applying Picone's identity for $p$-Laplacian, we also have $|h|\leq \lambda_f$ if \eqref{eq:M} admits a positive solution (see, e.g., \cite{Maliki,Ngo1}).

If we assume that $f$ does not change sign in the sense that $f\leq0$ but not strictly negative in $M$, or $\sup_Mf<0$, we then obtain
the following result.
\begin{theorem} \label{theorem2}
Let $(M, g)$ be a smooth compact Riemannian manifold of dimension $n$ $(n \geq 3)$ without boundary. Let $h<0$ be a constant,
$f$ and $a$ be smooth functions on $M$ with $a\geq 0$ in $M$ and $|h|<\lambda_f$. Moreover, we assume one of the following conditions holds:
\begin{itemize}
\item[\textup{(1)}]$f\leq0$ but not strictly negative;
\item[\textup{(2)}] $\sup_Mf<0$.
\end{itemize}
Then problem \eqref{eq:M} possesses a positive solution $u\in C^{1,\alpha}(M)$ for some $\alpha \in (0, 1)$.
\end{theorem}

Let us point out that in Theorem \ref{theorem2} for the case $p=2$, $|h|<\lambda_f$ is a necessary and sufficient solvability condition such that
Eq. \eqref{eq:M} admits a positive solution (see \cite{Ngo1} for more details). However, due to the quasilinear case $1<p<n$, it is difficult to obtain the same necessary and sufficient condition $|h|<\lambda_f$ as in the case $p=2$. Instead of  $|h|<\lambda_f$, we can only get $|h|\leq\lambda_f$ as the necessary condition in Theorem \ref{theorem2}.  We will focus on this problem in the future work.
Now, we describe the proof of our results briefly.
Due to the presence of a term with critical exponent and a term with a negative power, we first investigate the following $\varepsilon$-approximating subcritical equation:
\begin{equation}\label{eq:M2}
\Delta_{p,g}u+h|u|^{p-2}u=f(x)|u|^{q-2}u+\frac{a(x)u}{(u^{2}+\varepsilon)^{\frac{q}{2}+1}},
\end{equation}
for $\varepsilon >0$ is small and $q\in (p, p^*)$ is sufficiently close to $p^*$.
Based on Mountain Pass Lemma and minimization method, we obtain the existence results for \eqref{eq:M2}.
With the aid of the subcritical approach from variational method, we will show that solutions of \eqref{eq:M} exist as first $\varepsilon\searrow0$
and then $q\nearrow p^*$ under some given assumptions. We should also mention that though our method is partly similar as the arguments of Ng\^{o} and Xu (\cite{Ngo1}),
some technical difficulties are completely different in the quasilinear setting.
Moreover, compared with the results for $p=2$, our study on $p$-Lichnerowicz equation is generally harder.

The rest of this paper is organized as follows. In Section 2, we give some notations and prove some basic properties of solutions, including regularity and a non-existence result.
In Section 3, we perform analysis for the energy functional associated to \eqref{eq:M2}. In Section 4, we prove Theorem \ref{theorem1} and finally in Section 5, we complete the proof of Theorem \ref{theorem2}.

\section{Preliminary}
Let $(M, g)$ be a smooth compact Riemannian manifold of dimension $n\geq3$.  For simplicity,
we assume that the manifold $M$ has \textbf{unit volume}, i.e., $\textup{Vol}_g(M)=1$.
Let $L^p(M)$ for $1<p<n$ be the usual Lebesgue space on $(M, g)$.
For simplicity, we denote by $\|\cdot\|_p$ the $L^p$-norm, that is, $\|u\|_p=\left(\int_M|u|^p\,dv_g\right)^{1/p}$ for any $u\in L^p(M)$.
The Sobolev space $H_1^{p}(M)$ is defined as the completion of $C^{\infty}(M)$ with respect to the Sobolev norm
\begin{equation*}
\|u\| = \Big(\int_{M}|\nabla_{g} u|_g^{p}\,dv_g+\int_{M}|u|^{p}\,dv_g\Big)^{\frac{1}{p}}.
\end{equation*}
By the well-known Sobolev inequality, we know that for any $\varepsilon >0$,
there exists a constant $A=A(p, \varepsilon)$ such that for any $u \in H_1^{p}(M)$,
\begin{equation}\label{eq2.1}
\|u\|_{p^*}^{p}\leq (K(n,p)^p+\varepsilon)\|\nabla_{g} u\|_p^{p}+A\|u\|_p^{p},
\end{equation}
where $K(n,p)$ is the best constant for the embedding of $H_1^p(\mathbb{R}^n)$ into $L^{p^*}(\mathbb{R}^n)$, that is,
$$K(n,p)^{-1}=\inf_{u\in C^{\infty}_c (\mathbb{R}^n)\backslash\{0\}}\frac{\|\nabla u\|_{p}}{\|u\|_{p^*}},$$
where the norms here are corresponding to the Euclidean metric. For the explicit values of $K(n, p)$, we refer to Aubin \cite{Aubin0} and Talenti \cite{Tal76}.
Since we are interested in the critical case, throughout this paper, we always assume $q\in (p^{\flat},p^*)$, where
$p^{\flat}=\frac{p+p^*}{2}=\frac{p(2n-p)}{2(n-p)}$ is a dimensional constant. Let $f(x)$ be a smooth function on $M$, $f^-= \min\{f, 0\}$ and $f^+=\max\{f,0\}$.
We define the following two numbers
\begin{equation}\label{eq2.2}
\lambda_f=
\begin{cases}
\inf\limits_{u\in \mathcal{A}}\frac{\int_{M}|\nabla_g u|_g^{p}\,dv_g}{\int_{M}|u|^{p}\,dv_g} & \text{if }\mathcal{A}\neq \emptyset, \\
+\infty & \text{if }\mathcal{A}= \emptyset,
\end{cases}
\end{equation}
with
\begin{equation}\label{eq2.3}
\mathcal{A}=\left\{u\in H_1^p(M):u\geq0, u\not\equiv 0, \int_{M}|f^-|u^{p-1}\,dv_g=0\right\}.
\end{equation}
For $q\in (p^{\flat},p^*)$ and $\eta >0$, we define
\begin{equation}\label{eq2.4}
\lambda_{f,\eta,q}=\inf_{u\in \mathcal{A}_{\eta,q}}\frac{\int_{M}|\nabla_g u|_g^{p}\,dv_g}{\int_{M}|u|^{p}\,dv_g}
\end{equation}
with
\begin{equation}\label{eq2.5}
\mathcal{A}_{\eta,q}=\left\{u\in H_1^p(M):\|u\|_q=1, \int_{M}|f^-||u|^{q}\,dv_g=\eta\int_{M}|f^-|\,dv_g\right\}.
\end{equation}
Obviously, both $\lambda_f$ and $\lambda_{f,\eta,q}$ are non-negative. Moreover, the elements in $\mathcal{A}$ are regarded as functions that vanish on the support of $f^-$.
Similar to the case of $p=2$, the number $\lambda_f$ will play an important role on solving \eqref{eq:M}.
We will approximate $\lambda_f$ by $\lambda_{f,\eta,q}$ as proposed in \cite{Ngo1} and \cite{Rauzy} in Section \ref{lfeq} below.

For $p=2$, it is well-known that the solutions to \eqref{eq:M} is of $C^{\infty}(M)$. For $p\neq2$, since the $p$-Laplacian is degenerate at points where $\nabla_gu=0$,
the regularity of weak solutions to \eqref{eq:M2} is in general of $C^{1, \alpha}(M)$ for some $\alpha\in (0,1)$ but not of $C^2(M)$ (for instance, see \cite{Veron89, Tolk84}).
Inspired by the ideas from \cite{Ngo1, Druet}, we have the following regularity result.
\begin{lemma} \label{lemma 2.1}
Let $p<q\leq p^*$, $\varepsilon\geq 0$ fixed and $u\in H_1^p(M)$  be a weak solution of \eqref{eq:M2}.
Then we have:
\begin{itemize}
\item[\textup{(i)}] If $\varepsilon>0$, then $u\in C^{1, \alpha}(M)$ for some $\alpha \in (0, 1)$.

\item[\textup{(ii)}] If $\varepsilon=0$ and $u^{-1}\in L^r(M)$ for all $r\geq 1$, then $u\in C^{1, \alpha}(M)$ for some $\alpha \in (0, 1)$.
\end{itemize}
\end{lemma}
\begin{proof}(i) We first rewrite \eqref{eq:M2} as $\Delta_{p,g}u+\tilde{h}(x,u)=0$, where
\begin{equation*}
\tilde{h}(x,u)=h|u|^{p-2}u-f(x)|u|^{q-2}u-\frac{a(x)u}{(u^2+\varepsilon)^{\frac{q}{2}+1}}.
\end{equation*}
Notice that $p<q\leq p^*$, and then we have
\begin{equation*}
|\tilde{h}(x,u)|\leq C_1|u|^{p^*-1}+C_2
\end{equation*}
for some positive constants $C_1$, $C_2$.
Accordingly, by regularity result \cite[Theorem 2.3]{Druet}, we get that $u\in C^{1, \alpha}(M)$ for some $\alpha \in (0, 1)$.

(ii) We rewrite \eqref{eq:M2} as $\Delta_{p,g}u+K(x)|u|^{p-2}u=\tilde{f}(x)$, with
\begin{equation*}
K(x)=h-f(x)|u|^{q-p}\text{\quad and\quad} \tilde{f}(x)=\frac{a(x)}{|u|^{q}u}.
\end{equation*}
By the Sobolev embedding and the fact that $\frac{q}{q-p}\geq\frac{n}{p}$ we have
\begin{equation*}
|u|^{q-p}\in L^{\frac{n}{p}}(M) \text{\quad and\quad} K(x) \in L^{\frac{n}{p}}(M).
\end{equation*}
Due to the assumptions in case (ii), we have of course
$\tilde{f}(x) \in L^{\frac{n}{p}}(M)$. Accordingly, by regularity results (\cite{Druet}), we get that $u\in C^{1, \alpha}(M)$ for some $\alpha \in (0, 1)$.
\end{proof}

In order to avoid the lack of regularity, let us first consider the following nondegenerate  equation
\begin{equation}\label{eq2.7}
-\operatorname{div}_g((\eta+|\nabla_g u|_g^2)^{\frac{p-2}{2}}\nabla_g u)=g_{\varepsilon}\quad \text{in } M
\end{equation}
for a parameter $\eta>0$, where $g_{\varepsilon}=-h|u|^{p-2}u+f(x)|u|^{q-2}u+\frac{a(x)u}{(u^{2}+\varepsilon)^{\frac{q}{2}+1}}$
with the same assumptions on $f$, $h$ and $a$ as before. Then Equation \eqref{eq:M2} corresponds to the degenerate case $\eta=0$.
Since \eqref{eq2.7} is uniformly elliptic without singularities and the right hand side being
$C^1$-continuous, the solutions $u_{\eta}$ are in $C^{2, \delta}(M)$ for some $\delta\in (0,1)$ and the existence of
$u_{\eta}$ is also ensured by the classical theory (see \cite{Lady68}).
With these information in hand we have the following result which plays an important role on the proof of main results.
\begin{lemma} \label{lemma 2.2}
Let $u\in C^{1, \alpha}(M)$  be a positive solution of \eqref{eq:M2} with $h<0$.
Then, there holds
\begin{equation*}
\min_M{u}\geq \min\Big\{\Big(\frac{h}{\inf_M{f}}\Big)^{\frac{1}{p^{\flat}-p}},1\Big\}>0
\end{equation*}
for any $q\in (p^{\flat}, p^*)$ and any $\varepsilon>0$.
\end{lemma}
\begin{proof}
Let $u_{\eta}$ be a positive classical solution in $C^{2,\delta}(M)$ to \eqref{eq2.7}. From \cite{Veron89, Tolk84}, $u_{\eta}$
is bounded in $C^{1, \alpha}(M)$ independently of $\eta \in (0,1]$ and thus,
up to a subsequence, $u_{\eta}$ converges to $u$ in
$C^{1, \omega}(M)$ as $\eta\rightarrow 0$ for any $0<\omega<\alpha$.
Let us assume that $u_{\eta}$ achieves its minimum value at $x_{\eta}$. Notice that $u_{\eta}(x_{\eta})>0$ since $u_\eta(x)$ is a positive solution. We then have $\nabla_gu_{\eta}|_{x_{\eta}}=0$ and $\Delta_gu_{\eta}|_{x_{\eta}}\leq 0$. In particular, we have $$-\operatorname{div}_g((\eta+|\nabla_g u|_g^2)^{\frac{p-2}{2}}\nabla_g u)\big|_{x_{\eta}}\leq 0.$$
 Hence, we obtain
$$
h(u_{\eta}(x_{\eta}))^{p-2}u\geq f(x_{\eta})(u_{\eta}(x_{\eta}))^{q-1}+\frac{a(x_{\eta})u_{\eta}(x_{\eta})}{((u_{\eta}(x_{\eta}))^{2}+\varepsilon)^{\frac{q}{2}+1}}
\geq f(x_{\eta})(u_{\eta}(x_{\eta}))^{q-1}.
$$
Consequently, we get $f(x_{\eta})<0$ and thus $0<\frac{h}{f(x_{\eta})}\leq (u_{\eta}(x_{\eta}))^{q-p}$ which immediately implies
$$
\min_Mu_{\eta}\geq \left(\frac{h}{\inf_Mf}\right)^{\frac{1}{q-p}}\geq
\min\left\{\left(\frac{h}{\inf_Mf}\right)^{\frac{1}{p^{\flat}-p}},1\right\}
$$
for any $q\in (p^{\flat}, p^*)$. Now, taking $\eta\rightarrow 0$, we get the desired result.
\end{proof}

In the rest of this section, we derive a necessary condition for $a(x)$
such that the $p$-Laplacian Lichnerowicz equation \eqref{eq:M} admits no solution with finite $H_1^p$-norm.
Similar results can be found in \cite{Hebey3, Ngo1, Ngo3} for the case $p=2$.
\begin{proposition} \label{prop 2.3}
Let $(M, g)$ be a smooth compact Riemannian manifold without boundary of dimension $n \geq 3$. Let also $a(x)$, $f(x)$ be smooth functions on $M$
with $a(x)\geq 0$ in $M$ and $h$ a negative constant. If
\begin{equation*}
\int_M a^{\frac{np}{2np+n-p}}\,dv_g> \left(K(n,p)^p+1+A\right)^{\frac{2pn^2}{(2np+n-p)(n-p)}}{\Lambda}^{\frac{2p^2n^2}{(2np+n-p)(n-p)}}
\left(\int_M |f^-|^{p^*}\,dv_g\right)^{\frac{n-p}{2np+n-p}}
\end{equation*}
for some $\Lambda>0$, then the $p$-Laplacian Lichnerowicz equation \eqref{eq:M} has no positive solution $u$ with
energy $\|u\|\leq \Lambda$.
\end{proposition}
\begin{proof}
Let $u$ be a positive solution of \eqref{eq:M}.
By integrating \eqref{eq:M} over $M$ and applying the divergence theorem, we have
\begin{equation}\label{eq2.2-1}
\int_M h(x)u^{p-1}\,dv_g=\int_M f(x)u^{p^{*}-1}\,dv_g+\int_M\frac{a(x)}{u^{p^*+1}}\,dv_g.
\end{equation}
Let $\beta=\frac{p^*}{2p^*+1}$. Using H\"{o}lder's inequality, we obtain
\begin{equation}\label{eq2.2-2}
\int_M a^{\beta}\,dv_g\leq \left(\int_M \frac{a}{u^{p^*+1}}\,dv_g\right)^{\beta}\left(\int_M u^{p^*}\,dv_g\right)^{1-\beta}.
\end{equation}
For the second term of the right-hand side of \eqref{eq2.2-1}, notice that $h<0$ and we get
\begin{equation}\label{eq2.2-3}
\int_M \frac{a}{u^{p^*+1}}\,dv_g=\int_M hu^{p-1}\,dv_g-\int_M fu^{p^{*}-1}\,dv_g\leq \int_M |f^-|u^{p^{*}-1}\,dv_g,
\end{equation}
while for the first term, we obtain immediately, by H\"{o}lder's inequality,
\begin{equation}\label{eq2.2-4}
\int_M |f^-|u^{p^{*}-1}\,dv_g\leq \left(\int_M |f^-|^{p^*}\,dv_g\right)^{\frac{1}{p^*}}\left(\int_M u^{p^*}\,dv_g\right)^{\frac{p^*-1}{p^*}}.
\end{equation}
Combining \eqref{eq2.2-1}-\eqref{eq2.2-4}, we finally have that
\begin{equation}\label{eq2.2-5}
\int_M a^{\beta}\,dv_g\leq\left(\int_M |f^-|^{p^*}\,dv_g\right)^{\frac{\beta}{p^*}}\left(\int_M u^{p^*}\,dv_g\right)^{1-\frac{\beta}{p^*}}.
\end{equation}
Now, suppose that $\|u\|\leq \Lambda$. By Sobolev inequality \eqref{eq2.1} with $\varepsilon=1$
and the fact that $1-\frac{\beta}{p^*}=\frac{2p^*}{2p^*+1}$, we deduce that
\begin{equation*}
\left(\int_M u^{p^*}\,dv_g\right)^{1-\frac{\beta}{p^*}}\leq\left(K(n,p)^p+1+A\right)^{\frac{2(p^*)^2}{(2p^*+1)p}}\|u\|^{\frac{2(p^*)^2}{2p^*+1}}.
\end{equation*}
This together with \eqref{eq2.2-5} implies
\begin{equation*}
\int_M a^{\frac{p^*}{2p^*+1}}\,dv_g\leq\left(K(n,p)^p+1+A\right)^{\frac{2(p^*)^2}{(2p^*+1)p}}{\Lambda}^{\frac{2(p^*)^2}{2p^*+1}}
\left(\int_M |f^-|^{p^*}\,dv_g\right)^{\frac{1}{2p^*+1}},
\end{equation*}
which is a contradiction to our assumption. This completes the proof.
\end{proof}
\begin{remark} \rm
Proposition \ref{prop 2.3} shows that it is reasonable and necessary to have some control on the integral $\int_M a\,dv_g$ as we did in Theorem \ref{theorem1}.
Moreover, concerning \eqref{eq2.2-5} and Proposition \ref{prop 2.3}, as in \cite{Ngo3}, one can estimate the integral $\int_M a^{\alpha}|f^-|^{\beta}\,dv_g$
from above in terms of $\|u\|$, where $\alpha$, $\beta$ are two positive constant.
This also enables us to establish a sufficient condition to guarantee the nonexistence of positive solutions of \eqref{eq:M}.
\end{remark}

\section{The analysis of the energy functionals}\label{sec3}
Throughout this section, we always assume that $\sup_M f>0$.
For each $q\in (p,p^*)$ and $k>0$, we introduce ${\mathcal{B}}_{k,q}$ a hyper-surface of $H_1^p(M)$ which is defined as
\begin{equation*}
{\mathcal{B}}_{k,q}=\left\{u\in H_1^p(M): \|u\|_{q}^q=k\right\}.
\end{equation*}
Clearly, the set ${\mathcal{B}}_{k,q}$ is non-empty for any $k>0$. Now we construct the approximated energy functional associated to subcritical problem \eqref{eq:M2}.
For each $\varepsilon>0$, we define the functional $\mathcal{I}_q^{\varepsilon}:H_1^p(M)\rightarrow \mathbb{R}$ as
\begin{equation*}
{\mathcal{I}}_q^{\varepsilon}(u)=\frac{1}{p}\int_{M}|\nabla_{g} u|_g^{p}\,dv_g+\frac{h}{p}\int_{M}|u|^p\,dv_g
-\frac{1}{q}\int_{M}f|u|^q\,dv_g+\frac{1}{q}\int_{M}\frac{a}{(u^{2}+\varepsilon)^{\frac{q}{2}}}\,dv_g.
\end{equation*}
By a standard argument, we have that $\mathcal{I}_q^{\varepsilon} \in C^1(H_1^p(M),\mathbb{R})$. Let $\delta \mathcal{I}_q^{\varepsilon}$ be the first variation of $\mathcal{I}_q^{\varepsilon}$, namely,
\begin{equation*}
\begin{split}
\delta \mathcal{I}_q^{\varepsilon}(u)(\varphi)=&\int_M |\nabla_gu|^{p-2}g(\nabla_gu,\nabla_g\varphi)\,dv_g+h\int_{M}|u|^{p-2}u\varphi\,dv_g\\
&-\int_{M}f|u|^{q-2}u\varphi\,dv_g-\int_{M}\frac{au\varphi}{(u^{2}+\varepsilon)^{\frac{q}{2}+1}}\,dv_g \,\, \text{  for all } \varphi\in H_1^p(M).
\end{split}
\end{equation*}
Therefore, weak solutions of \eqref{eq:M2} correspond to critical points of $\mathcal{I}_q^{\varepsilon}$.
Set
\begin{equation*}
\mu_{k,q}^{\varepsilon}=\inf_{u\in {\mathcal{B}}_{k,q}}\mathcal{I}_q^{\varepsilon}(u).
\end{equation*}
By H\"{o}lder's inequality and the fact that $\mathrm{Vol}_g(M)=1$, it holds
$\mathcal{I}_q^{\varepsilon}(u)\geq \frac{h}{p}k^{\frac{p}{q}}-\frac{k}{q}\sup_Mf$ for any $u\in {\mathcal{B}}_{k,q}$.
From this we know that $\mu_{k,q}^{\varepsilon}> -\infty$ provided that $k$ is finite. On the other hand, using the test function $u=k^{\frac{1}{q}}$, we obtain
\begin{equation}\label{eq3.1-3}
\mu_{k,q}^{\varepsilon}\leq\frac{h}{p}k^{\frac{p}{q}}-\frac{k}{q}\int_{M}f\,dv_g+\frac{1}{q}\int_{M}\frac{a}{(k^{\frac{2}{q}}
+\varepsilon)^{\frac{q}{2}}}\,dv_g,
\end{equation}
which implies that $\mu_{k,q}^{\varepsilon}<+\infty$.
\subsection{The asymptotic behavior of \texorpdfstring{$\mu_{k,q}^{\varepsilon}$}{mu\_{k,q} varepsilon}}\label{subsec3.2}

In this subsection, we first show that if $k$, $q$ and $\varepsilon$ are fixed, then $\mu_{k,q}^{\varepsilon}$ is achieved
by some positive function, say $\bar{u}$.
\begin{lemma} \label{lemma 3.2}
$\mu_{k,q}^{\varepsilon}$ is attained by a positive function $\bar{u}\in C^{1, \alpha}(M)$ for $q<p^*$.
\end{lemma}
\begin{proof}
Indeed, let $\{u_j\}_{j}$ be a minimizing
sequence for $\mu_{k,q}^{\varepsilon}$, that is,
$$u_j \in {\mathcal{B}}_{k,q}\quad \text{and} \quad \mathcal{I}_q^{\varepsilon}(u_j)\rightarrow \mu_{k,q}^{\varepsilon}.$$
Since $\mathcal{I}_q^{\varepsilon}(u_j)=\mathcal{I}_q^{\varepsilon}(|u_j|)$, we may assume that $u_j\geq 0$ for all
$j\geq 1$. By H\"{o}lder's inequality, one has $\|u_j\|_p\leq k^{\frac{1}{q}}$. Since
$\mathcal{I}_q^{\varepsilon}(u_j)\leq \mu_{k,q}^{\varepsilon}+1$ for sufficiently large $j$, it holds
$$\frac{1}{p}\int_{M}|\nabla_{g} u_j|_g^{p}\,dv_g\leq \mu_{k,q}^{\varepsilon}-\frac{h}{p}k^{\frac{p}{q}}+\frac{k}{q}\sup_Mf+1.$$
Hence,  the sequence $\{u_j\}_{j}\subset H_1^p(M)$ is bounded and, up to subsequences,
\begin{equation*}
  \begin{split}
  & u_j \rightharpoonup \bar{u}   \,\, \text{ weakly in }    H_1^p(M), \,\,  u_j \to \bar{u}  \,\, \text{ strongly in } L^q(M),\,\,\text{and} \\
  & u_j(x) \to \bar{u}(x)  \,\, \text{ a.e. in } M \,\, \text{ as }j \to +\infty.
  \end{split}
 \end{equation*}
This shows that $\bar{u}(x)\geq 0$ a.e. on $M$ and $\|\bar{u}\|_q=k^{\frac{1}{q}}$.
In particular, we have $\bar{u}\in {\mathcal{B}}_{k,q}$.
Now noticing that $a\varepsilon^{-\frac{q}{2}}\in L^1(M)$, we obtain by Lebesgue's dominated convergence theorem that
$$\int_{M}\frac{a}{(u_j^{2}+\varepsilon)^{\frac{q}{2}}}\,dv_g\rightarrow \int_{M}\frac{a}{({\bar{u}}^{2}+\varepsilon)^{\frac{q}{2}}}\,dv_g \,\, \text{ as } j \to +\infty.$$
 Hence, from the weak lower semi-continuity of the integral functionals, we get
$$\mu_{k,q}^{\varepsilon}=\lim\limits_{j \to +\infty}\mathcal{I}_q^{\varepsilon}(u_j)\geq \mathcal{I}_q^{\varepsilon}(\bar{u}).$$
This and the fact that $\bar{u}\in {\mathcal{B}}_{k,q}$ immediately give us
$\mu_{k,q}^{\varepsilon}= \mathcal{I}_q^{\varepsilon}(\bar{u})$.

Next, we will show the regularity and positivity of $\bar{u}$.
Invoked by the Lagrange multiplier rule, we can find $\lambda \in \mathbb{R}$, such that $\bar{u}$ solves
\begin{equation}\label{eq:3.2-1}
\Delta_{p,g}\bar{u}+h|\bar{u}|^{p-2}\bar{u}=(f(x)+\lambda)|\bar{u}|^{q-2}\bar{u}+\frac{a(x)\bar{u}}{(\bar{u}^{2}+\varepsilon)^{\frac{q}{2}+1}}
\end{equation}
in the weak sense. It follows from Lemma \ref{lemma 2.1} that $\bar{u}\in C^{1, \alpha}(M)$ for some
$\alpha\in (0,1)$ and $\bar{u}\geq 0$ in $M$. Furthermore, applying the strong maximum principle (see \cite[Theorem 2.6]{Druet}) and noticing that
$\int_M(\bar{u})^q\,dv_g=k\neq 0$, we conclude that $\bar{u}>0$. Thus, $\bar{u}$ is a positive solution of
\eqref{eq:3.2-1}.
\end{proof}

The following interesting property of $\mu_{k,q}^{\varepsilon}$ will also be used in the proofs of our main results.
\begin{proposition}\label{prop3.1}
For $\varepsilon>0$ fixed, $\mu_{k,q}^{\varepsilon}$ is continuous with respect to $k$.
\end{proposition}
\begin{proof}
First we know that $\mu_{k,q}^{\varepsilon}$ is well-defined for any $k\in (0,+\infty)$. We have to verify that for each $k$ fixed and for
any sequence $k_j\rightarrow k$ there holds $\mu_{k_j,q}^{\varepsilon}\rightarrow \mu_{k,q}^{\varepsilon}$ as $j\rightarrow +\infty$.
This is equivalent to show that there exists a subsequence of $\{k_j\}_j$, still denoted by $\{k_j\}_j$, such that
$\mu_{k_j,q}^{\varepsilon}\rightarrow \mu_{k,q}^{\varepsilon}$ as $j\rightarrow +\infty$.
We suppose that $\mu_{k_j,q}^{\varepsilon}$ and $ \mu_{k,q}^{\varepsilon}$ are achieved by $u_j \in {\mathcal{B}}_{k_j,q}$ and $u \in {\mathcal{B}}_{k,q}$
respectively. From Lemma \ref{lemma 3.2}, $u_j$ and $u$ are positive functions on $M$. We need to prove the boundedness of $u_j$ in $H_1^p(M)$.
It then suffices to control $\|\nabla_gu\|_p$. In fact, as in \eqref{eq3.1-3}, one has
\begin{equation}\label{e3.7}
\int_{M}|\nabla_{g} u_j|_g^{p}\,dv_g<p\Big(\mu_{k_j,q}^{\varepsilon}-\frac{h}{p}k_j^{\frac{p}{q}}+\frac{k_j}{q}\sup_M f+1\Big).
\end{equation}
By the homogeneity we can find a sequence of positive numbers $\{t_j\}_j$ such that
$t_ju\in {\mathcal{B}}_{k_j,q}$. Since $k_j\rightarrow k$ as $j\rightarrow +\infty$ and $k_j^{{1}/{q}}=\|t_ju\|_q=t_jk^{{1}/{q}}$,
we immediately see that $t_j\rightarrow 1$ as $j\rightarrow +\infty$. Now, substituting $u$ by $t_ju$ in ${\mathcal{I}}_q^{\varepsilon}(u)$, it holds
\begin{equation}\label{e3.8}
\begin{split}
\mu_{k_j,q}^{\varepsilon} \leq\,\, &  t_j^p\Big(\frac{1}{p}\int_M |\nabla_gu|_g^{p}\,dv_g+\frac{h}{p}\int_{M}|u|^{p}\,dv_g\Big)\\
& -\frac{t_j^q}{q}\int_{M}f|u|^{q}\,dv_g+\frac{1}{q}\int_{M}\frac{a}{(u^{2}+\varepsilon)^{\frac{q}{2}}}\,dv_g.
\end{split}
\end{equation}
Notice that $u$ is fixed and $t_j$ belongs to a neighborhood of $1$ for large $j$. Thus, $\{\mu_{k_j,q}^{\varepsilon}\}_j$ is bounded which also implies by
\eqref{e3.7} that $\{\|\nabla_gu_j\|_p\}_j$ is bounded. Hence, $\{u_j\}_j$ is bounded in $H_1^p(M)$. Consequently, there exists a $\bar{u}\in H_1^p(M)$ such that,
up to subsequences, $u_j\rightarrow \bar{u}$ strongly in $L^r(M)$ for any $r\in [1, p^*)$,
and $\lim\limits_{j\rightarrow+\infty}\|u_j\|_q=\|\bar{u}\|_q=k^{\frac{1}{q}}$,
that is, $\bar{u}\in {\mathcal{B}}_{k,q}$.
Thus, we have $\mathcal{I}_q^{\varepsilon}(u)=\mu_{k,q}^{\varepsilon}\leq \mathcal{I}_q^{\varepsilon}(\bar{u})$.
By weak lower semi-continuity property of $\mathcal{I}_q^{\varepsilon}$, we deduce that
$$\mathcal{I}_q^{\varepsilon}(u)\leq \mathcal{I}_q^{\varepsilon}(\bar{u})\leq \liminf_{j\rightarrow +\infty}\mathcal{I}_q^{\varepsilon}(u_j).$$
On the other hand, by \eqref{e3.8} and Lebesgue's dominated convergence theorem, it holds
$\limsup\limits_{j\rightarrow +\infty}\mu_{k_j,q}^{\varepsilon}\leq \mathcal{I}_q^{\varepsilon}(u)$
since $t_j\rightarrow 1$ as $j\rightarrow +\infty$.
Therefore, $\lim\limits_{j\rightarrow +\infty}\mu_{k_j,q}^{\varepsilon}=\mu_{k,q}^{\varepsilon}$. This completes the proof.
\end{proof}

The following Lemma describe the asymptotic behavior of $\mu_{k,q}^{\varepsilon}$ as $k$ varies (See Figure \ref{Figure1}).
We omit its proof since the proof is similar as the case $p=2$ in \cite{Ngo1}.
\begin{figure}[H]
  \centering
  \includegraphics[width=7cm]{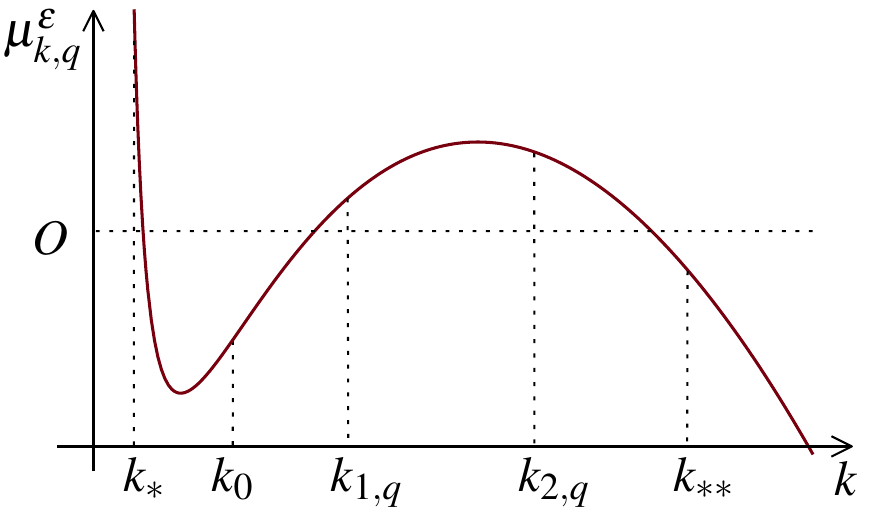}\\
  \caption{The asymptotic behavior of $\mu_{k,q}^{\varepsilon}$ when $\sup\limits_Mf>0$.}\label{Figure1}
\end{figure}
\begin{lemma}\label{asymptotic1}
 The following asymptotic behaviors hold:
 \begin{itemize}
  \item [\textnormal{(i)}] $\mu_{k,q}^{k^{\frac{2}{q}}}\rightarrow +\infty$ as $k\rightarrow 0^+$. In particular, there exists $k_*$ sufficiently small
and independent of both $q$ and $\varepsilon$ such that $\mu_{k_*,q}^{\varepsilon}>0$ for any $\varepsilon\leq k_*$.
  \item [\normalfont\text{(ii)}] $\mu_{k,q}^{\varepsilon}\rightarrow -\infty$ as $k\rightarrow +\infty$ provided $\sup\limits_Mf>0$.
  \item [\normalfont\text{(iii)}] There exists
$k_0=\left(\frac{(p+q)}{2p}\frac{|h|}{\int_M|f^-|\,dv_g}\right)^{{q}/{(q-p)}}$
such that $\mu_{k_0,q}^{\varepsilon}\leq0$ for any
$\varepsilon>0$ provided
\begin{equation} \label{e3.5}
\int_{M}a\,dv_g\leq \Big(\frac{p+q}{2p}\frac{|h|}{\int_{M}|f^-|dv_g}\Big)^{\frac{q+p}{q-p}}\frac{|h|}{2p}(q-p).
\end{equation}
In particular, $k_0>k_*$.
 \item [\textnormal{(iv)}]  Assume that \eqref{e1.3} holds. Then there exists some constant $\mu$ independent of $q$ and $\varepsilon$ such that
$\mu_{k,q}^{\varepsilon}\leq\mu$ for any $\varepsilon>0$, $q\in(p^{\flat},p^*)$ and $k\geq k_0$.
  \item [\normalfont\text{(v)}] There is some $k_{**}$ sufficiently large and independent of both $q$ and $\varepsilon$ such that $\mu_{k,q}^{\varepsilon}<0$ for any $k\geq k_{**}$.
 \end{itemize}
 \end{lemma}

The following remarks will be needed in future argument.
\begin{remark}\label{rm3.3}
A simple calculation shows that the following function
\begin{equation*}
\phi(q):=\Big(\frac{p+q}{2p}\frac{|h|}{\int_{M}|f^-|dv_g}\Big)^{\frac{q+p}{q-p}}\frac{|h|}{2p}(q-p)
\end{equation*}
is monotone increasing in $(p,p^*)$ provided that $\frac{p^*}{p}|h|\leq \int_{M}|f^-|dv_g$.
Thus, the term on the right hand side of \eqref{e1.3} equals $\lim\limits_{q\to p^*}\phi(q)$. This suggests us that a good condition
for $\int_{M}adv_g$ could be \eqref{e1.3}.
\end{remark}
\begin{remark}\label{rm3.4}
It follows from $q\in (p^{\flat},p^*)$ that
\begin{equation*}
\min\left\{\left(\frac{|h|}{\int_{M}|f^-|dv_g}\right)^{\frac{2n-p}{p}}, 1\right\}\leq k_0,
\end{equation*}
where $k_0$ is given in Lemma \ref{asymptotic1}(iii), since $\frac{p+q}{2p}>1$ and the function $\frac{q}{q-p}$ is monotone decreasing.
Thus, as in \cite{Ngo1}, one can easily control the growth of $\mu_{k_0,q}^{\varepsilon}$ as follows
\begin{equation}\label{e3.6.2}
\mu_{k_0,q}^{\varepsilon}\leq -\frac{1}{p^*}\min\left\{\left(\frac{|h|}{\int_{M}|f^-|dv_g}\right)^{\frac{2n-p}{p}}, 1\right\}\int_Mf^+\,dv_g
\end{equation}
for any $\varepsilon \geq 0$. The importance of \eqref{e3.6.2} is that the right hand side is strictly negative and independent on both $q$ and $\varepsilon$
provided that $\sup_Mf>0$, which is always the assumption in this section.
\end{remark}

The next subsection is originally due to Rauzy \cite[Subsection IV.3]{Rauzy} for prescribing the scalar curvature on a compact Riemannian manifold
of negative conformal invariant.
For the sake of clarity, we follow the argument in \cite{Rauzy} to re-prove \cite[Subsection IV.3]{Rauzy} for our quasilinear setting.
\subsection{The study of \texorpdfstring{$\lambda_{f, \eta, q}$}{lambda\_f, eta, q}}\label{lfeq}

The proof of our main result depends on $\lambda_{f, \eta, q}$ (cf. \eqref{eq2.4}). This quantity was first introduced by Rauzy \cite{Rauzy}.
Let us first define $\lambda_{f,\eta,q}'=\inf_{u\in \mathcal{A}_{\eta,q}'}{\|\nabla_g u\|_p^p}/{\|u\|_p^p}$, where
$$\mathcal{A}_{\eta,q}'=\{u\in H_1^p(M):\|u\|_q=1,\,\, \text{and }\int_{M}|f^-||u|^{q}\,dv_g \leq \eta\int_{M}|f^-|\,dv_g\}. $$
For $q$ and $\eta>0$ fixed, the set $\mathcal{A}_{\eta,q}'$ is not empty since it includes the set of functions $u\in H_1^p(M)$ such that $\|u\|_q=1$
and $\textup{supp} (u)\subset\{x\in M: f(x)>0\}$, and thus $\lambda_{f,\eta,q}'$ is finite.
Moreover, it is easy to check that $\lambda_{f,\eta,q}'$ is monotone decreasing with respect to $\eta$.
We are now going to prove $\lambda_{f,\eta,q}'=\lambda_{f,\eta,q}$. To this end, it suffices to prove $\lambda_{f,\eta,q}'\geq\lambda_{f,\eta,q}$
since the reverse is automatically true.

\begin{lemma}\label{lemma3.3}
As a function of $\eta>0$, $\lambda_{f,\eta,q}$ is monotone decreasing.
\end{lemma}
\begin{proof}
We first prove that $\lambda_{f,\eta,q}'$ can be achieved. Let $\{v_j\}_j\subset \mathcal{A}_{\eta,q}'$  be a minimizing sequence
for $\lambda_{f,\eta,q}'$. Obviously the sequence $\{|v_j|\}_j$ is still a minimizing sequence in $\mathcal{A}_{\eta,q}'$,
and then we can assume that $v_j\geq 0$ in $M$. By the same arguments as in Lemma \ref{lemma 3.2}, we obtain that $\{v_j\}_j$
is bounded in $H_1^p(M)$. Then up to subsequences, there exists $v\in  H_1^p(M)$ such that
\begin{align*}
\begin{split}
&v_j \rightharpoonup v \,\, \text{ weakly in } H_1^p(M),\,\, v_j \rightarrow v \,\, \text{ strongly in } L^{r}(M) \,\, \text{for all } r\in [1,p^*),\,\, \text{and }\\
&v_j(x) \rightarrow v(x) \,\, \text{ for almost every }x\in M \,\,\text{ as } j \to +\infty.
\end{split}
\end{align*}
Consequently, $v(x)\geq 0$ a.e. on $M$ and $\|v\|_q=1$.
Moreover, Lebesgue's dominated convergence theorem allows us to conclude that $\int_{M}|f^-||v|^{q}\,dv_g \leq \eta\int_{M}|f^-|\,dv_g$.
Hence, $v\in \mathcal{A}_{\eta,q}'$. We notice that
$$ \|v\|_p^p= \lim_{j\rightarrow +\infty}\|v_j\|_p^p   \quad \text{and} \quad \|\nabla_gv\|_p^p\leq \lim_{j\rightarrow +\infty}\|\nabla_gv_j\|_p^p.$$
It follows that $\|\nabla_gv\|_p^p/\|v\|_p^p\leq \lambda_{f,\eta,q}'$. Thus, $\lambda_{f,\eta,q}'$ is achieved by $v$.
Moreover, $\|\nabla_gv\|_p>0$ since we can assume that $\lambda_{f,\eta,q}'>0$, otherwise $\lambda_{f,\eta,q}'=0$, and this is trivial.
By \cite[Proposition 3.49]{Aubin1}, we may assume $v\geq 0$, otherwise we just replace $v$ by $|v|$.
We now claim that $v\in \mathcal{A}_{\eta,q}$, where the definition of $\mathcal{A}_{\eta,q}$ is in \eqref{eq2.5}.
Indeed, we assume by contradiction that $v\not\in \mathcal{A}_{\eta,q}$,
that is, $\int_{M}|f^-|v^{q}\,dv_g < \eta\int_{M}|f^-|\,dv_g$.
Then there exists a positive constant $\tau$ such that
$$\int_{M}|f^-|(v+\tau)^{q}\,dv_g = \eta\int_{M}|f^-|\,dv_g\,\, \text{ and }\,\, \|v+\tau\|_q\geq1.$$
It follows that $\frac{v+\tau}{\|v+\tau\|_q}\in \mathcal{A}_{\eta,q}'$, and then we deduce
$$\Big\|\nabla_g\Big(\frac{v+\tau}{\|v+\tau\|_q}\Big)\Big\|_p^p\Big\|\frac{v+\tau}{\|v+\tau\|_q}\Big\|_p^{-p}=\frac{\|\nabla_g({v+\tau})\|_p^p}{\|{v+\tau}\|_p^p}
<\frac{\|\nabla_gv\|_p^p}{\|v\|_p^p}=\lambda_{f,\eta,q}'.$$
which gives us a contradiction. Hence, $v\in \mathcal{A}_{\eta,q}$. In particular, we have $\lambda_{f,\eta,q}'=\lambda_{f,\eta,q}$.
Therefore, $\lambda_{f,\eta,q}$ is a decreasing function of $\eta$.
\end{proof}

Lemma \ref{lemma3.3} says that $\mathcal{A}_{\eta,q}\neq \emptyset$.
The following lemma gives us a comparison between $\lambda_{f,\eta,q}$ (cf. \eqref{eq2.4}) and $\lambda_{f}$ (cf. \eqref{eq2.2}).
Intuitively, $\mathcal{A}$ (cf. \eqref{eq2.3}) is smaller than
$\mathcal{A}_{\eta,q}$.

\begin{lemma}\label{lemma3.4}
For each $q\in (p, p^*)$ and $\eta>0$ fixed, it holds $\lambda_{f,\eta,q}\leq \lambda_f$.
\end{lemma}
\begin{proof}
Let $u\in \mathcal{A}$. Then there holds $\int_M u\,dv_g>0$, otherwise, $u\equiv0$. By the definition of $\mathcal{A}$,
we have $\int_{M}|f^-|u^{p-1}\,dv_g=0$ which also implies that $\int_{M}|f^-|u^q\,dv_g=0$. Again, from the definition of $\mathcal{A}$
and the fact that $\sup_M f>0$, we must have $\int_{M}u^q\,dv_g>0$. We now choose $\varepsilon>0$ such that $\int_{M}(\varepsilon u)^q\,dv_g=1$.
Then we have $\varepsilon u\in \mathcal{A}_{\eta,q}'$ and
$$\lambda_{f,\eta,q}'\leq\frac{\|\nabla_g(\varepsilon u)\|_p^p}{\|\varepsilon u\|_p^p}=\frac{\|\nabla_gu\|_p^p}{\|u\|_p^p}.$$
Since the preceding inequality holds for any $u\in \mathcal{A}$, we may take the infimum on both sides with respect to $u$ to
arrive at $\lambda_{f,\eta,q}=\lambda_{f,\eta,q}'\leq \lambda_f$.
\end{proof}

The next lemma shows that the number $\lambda_{f,\eta,q}$ can be close arbitrarily to $\lambda_{f}$.
\begin{lemma}\label{lemma3.5}
For each $\delta>0$ fixed, there exists $\eta_0>0$ such that for all $\eta<\eta_0$, there is a $q_{\eta}\in (p^{\flat}, p^*)$
so that $\lambda_{f,\eta,q}\geq \lambda_f-\delta$ for every $q\in (q_{\eta}, p^*)$.
\end{lemma}
\begin{proof}
By contradiction, assume that there is some $\delta_0>0$ such that for every $\eta_0>0$, there exist $\eta<\eta_0$
and a monotone sequence $\{q_j\}_j$ converging to $p^*$ so that $\lambda_{f,\eta,q_j}< \lambda_f-\delta_0$ for every $j$.
We can furthermore assume that $\lambda_{f,\eta,q_j}$ is achieved by $v_{\eta, q_j}\in \mathcal{A}_{\eta, q_j}$.
We then immediately have
\begin{equation}\label{e3.9}
\frac{\|\nabla_g{v_{\eta, q_j}}\|_p^p}{\|v_{\eta, q_j}\|_p^p}\leq\lambda_{f}-\delta_0
\end{equation}
for any $j$. Due to the finiteness of $\lambda_{f}$, we can prove the boundedness of $\{v_{\eta, q_j}\}_j$ in $H_1^p(M)$,
which helps us to select a subsequence of $\{v_{\eta, q_j}\}_j$ so that
\begin{equation*}
\begin{cases}
\,v_{\eta, q_j} \rightharpoonup v_{\eta, p^*} \quad \text{weakly in } H_1^p(M),\\
\,\nabla_gv_{\eta, q_j} \rightharpoonup \nabla_gv_{\eta, p^*} \quad \text{weakly in } L^{p}(M),\\
\,v_{\eta, q_j} \rightarrow v_{\eta, p^*} \quad \text{strongly in } L^{r}(M) \,\, \text{for all } r\in [1,p^*),\\
\,v_{\eta, q_j}(x) \rightarrow v_{\eta, p^*}(x) \quad \text{almost everywhere on } M,
\end{cases}
\end{equation*}
for some $v_{\eta, p^*}\in H_1^p(M)$ as $j \to +\infty$. Taking the limit in \eqref{e3.9}, we have
\begin{equation}\label{e3.10}
\frac{\|\nabla_g{v_{\eta, p^*}}\|_p^p}{\|v_{\eta, p^*}\|_p^p}\leq\lambda_{f}-\delta_0.
\end{equation}
Besides, H\"{o}lder's inequality tells us $1\leq \|v_{\eta, q_j}\|_{p^*}$ for each $j$,
which implies
\begin{equation*}\label{e3.11}
\begin{split}
1 &\leq  \big[(K(n,p)^p+1)\frac{\|\nabla_g{v_{\eta, q_j}}\|_p^p}{\|v_{\eta, q_j}\|_p^p}+A\big]\|v_{\eta, q_j}\|_p^p\\
& \leq \big[(K(n,p)^p+1)(\lambda_{f}-\delta_0)+A\big]\|v_{\eta, q_j}\|_p^p.
\end{split}
\end{equation*}
Then it yields
\begin{equation}\label{e3.12}
\frac{1}{(K(n,p)^p+1)(\lambda_{f}-\delta_0)+A}\leq \|v_{\eta, q_j}\|_p^p
\end{equation}
for each $j$. Passing to the limit as $j\rightarrow +\infty$ in \eqref{e3.12}, one obtains
\begin{equation}\label{e3.13}
\frac{1}{(K(n,p)^p+1)(\lambda_{f}-\delta_0)+A}\leq \|v_{\eta, p^*}\|_p^p=\int_M |v_{\eta, p^*}|^p\,dv_g.
\end{equation}
For every $q_j\geq p^{\flat}$, since $v_{\eta, q_j}\in \mathcal{A}_{\eta, q_j}$, it holds $\int_M |v_{\eta, q_j}|^{p^{\flat}}\,dv_g\leq 1$ and
\begin{equation*}\label{e3.14}
\begin{split}
\int_M |f^-||v_{\eta, q_j}|^{p^{\flat}}\,dv_g &\leq  \Big(\int_M |f^-||v_{\eta, q_j}|^{q_j}\,dv_g\Big)^{\frac{p^{\flat}}{q_j}}\Big(\int_M |f^-|\,dv_g\Big)^{1-\frac{p^{\flat}}{q_j}}\\
& ={\eta}^{\frac{p^{\flat}}{q_j}}\int_M |f^-|\,dv_g.
\end{split}
\end{equation*}
Here we use H\"{o}lder's inequality. Taking $j\rightarrow +\infty$, by the Fatou's lemma, we deduce that
$\int_M |v_{\eta, p^*}|^{p^{\flat}}\,dv_g\leq 1$ and
\begin{equation*}\label{e3.15}
\int_M |f^-||v_{\eta, p^*}|^{p^{\flat}}\,dv_g \leq {\eta}^{\frac{p^{\flat}}{p^*}}\int_M |f^-|\,dv_g.
\end{equation*}
Now we let $\eta_0\rightarrow 0$, and then clearly $\eta\rightarrow 0$. The boundedness of $v_{\eta, p^*}$
in $H_1^p(M)$ follows from the fact $v_{\eta, q_j} \rightharpoonup v_{\eta, p^*}$ weakly in $H_1^p(M)$
and $\lambda_f$ is finite. Therefore, there exists $v\in H_1^p(M)$ such that, up to subsequences,
\begin{equation}\label{e3.16}
\begin{cases}
\,v_{\eta, p^*} \rightharpoonup v \quad \text{weakly in } H_1^p(M),\\
\,v_{\eta, p^*} \rightarrow v \quad \text{strongly in } L^{r}(M) \,\, \text{for all } r\in [1,p^*),\\
\,v_{\eta, p^*}(x) \rightarrow v(x) \quad \text{almost everywhere on } M,
\end{cases}
\end{equation}
as $\eta \to 0$. Before giving out contradiction, we notice from \eqref{e3.10} that
\begin{equation}\label{e3.17}
\|\nabla_gv\|_p^p\leq(\lambda_{f}-\delta_0)\|v\|_p^p.
\end{equation}
Then it is enough to see
\begin{equation*}\label{e3.18}
\begin{split}
0\leq\int_M |f^-||v|^{p^{\flat}}\,dv_g &\leq \lim_{\eta\rightarrow0} \int_M |f^-||v_{\eta, p^*}|^{p^{\flat}}\,dv_g\\
& \leq \lim_{\eta\rightarrow0} \Big({\eta}^{\frac{p^{\flat}}{p^*}}\int_M |f^-|\,dv_g\Big)=0.
\end{split}
\end{equation*}
Thus, we have $\int_M |f^-||v|^{p^{\flat}}\,dv_g=0$. In particular, $\int_M |f^-||v|^{p-1}\,dv_g=0$.
By \eqref{e3.13} and \eqref{e3.16}, we obtain that
\begin{equation*}
\frac{1}{(K(n,p)^p+1)\lambda_{f}+A}\leq \int_M |v|^p\,dv_g.
\end{equation*}
Therefore, $v\not\equiv0$, and thus $|v|\in \mathcal{A}$. By the definition of $\lambda_f$, we know that
\begin{equation*}\label{e3.20}
\lambda_{f}\|v\|_p^p\leq\|\nabla_g|v|\|_p^p=\|\nabla_gv\|_p^p.
\end{equation*}
This contradicts \eqref{e3.17} and concludes the proof.
\end{proof}

Now, we prove that, for any $\varepsilon>0$ and some $k>k_0$, $\mu_{k,q}^\varepsilon>0$. The similar argument can be found in \cite[Proposition 2]{Rauzy} for the prescribed scalar curvature equation and \cite[Proposition 3.14]{Ngo1} for the Lichnerowicz equation.
\begin{proposition}\label{prop3.2}
Suppose that $\sup_Mf>0$ and $|h|<\lambda_f$. Then there exists $\eta_0>0$ and its
corresponding $q_{\eta_0}$ sufficiently close to $p^*$ such that
\begin{equation}\label{e3.21}
\delta:=\frac{\lambda_{f, \eta_0, q}+h}{p}>\frac{1}{2p}(\lambda_{f}+h)
\end{equation}
for any $q\in (q_{\eta_0}, p^*)$. For such $\delta$, we denote
\begin{equation}\label{e3.22}
\mathcal{C}_q=\frac{\eta_0}{4|h|}\underbrace{\min\left\{\frac{\delta}{A+(K(n,p)^p+1)(|h|+p\delta)}, \frac{(p-1)|h|}{p}\right\}}_{m}:=\frac{\eta_0}{4|h|}m.
\end{equation}
If
\begin{equation}\label{e3.23}
{\sup_Mf}{\left(\int_{M}|f^-|\,dv_g\right)^{-1}}<\mathcal{C}_q,
\end{equation}
then there exists an interval $I_q=[k_{1,q}, k_{2,q}]$ such that for any $k\in I_q$, any $\varepsilon>0$ and $u\in \mathcal{B}_{k,q}$,
there holds $\mathcal{I}_q^{\varepsilon}(u)>\frac{1}{2}mk^{\frac{p}{q}}$. In particular, $\mu_{k,q}^{\varepsilon}>0$ for any
$k\in I_q$ and any $\varepsilon>0$.
\end{proposition}
\begin{proof}
It follows from Lemma \ref{lemma3.5} that there exists some $0<\eta_0<2$ and its corresponding $q_{\eta_0}\in (p^{\flat}, p^*)$
such that
\begin{equation*}\label{e3.24}
0\leq\lambda_f-\lambda_{f, \eta_0, q}<\frac{1}{2}(\lambda_f-|h|)
\end{equation*}
for any $q\in (q_{\eta_0}, p^*)$. This immediately implies \eqref{e3.21}. Set
\begin{equation}\label{e3.25}
k_{1,q}=\left(\frac{|h|q}{\eta_0\int_M|f^-|\, dv_g}\right)^{\frac{q}{q-p}}.
\end{equation}
From Lemma \ref{asymptotic1} and the fact that $0<\eta_0<2$ we deduced that $k_0<k_{1,q}$.
We assume form now on that $k\geq k_{1,q}$. Let $u\in \mathcal{B}_{k,q}$. We write
\begin{equation*}\label{e3.26}
\mathcal{I}_q^{\varepsilon}(u)=G_q(u)-\frac{1}{q}\int_{M}f^+|u|^q\,dv_g+\frac{1}{q}\int_{M}\frac{a}{(u^{2}+\varepsilon)^{\frac{q}{2}}}\,dv_g,
\end{equation*}
where
\begin{equation}\label{e3.27}
G_q(u)=\frac{1}{p}\int_{M}|\nabla_{g} u|_g^{p}\,dv_g+\frac{h}{p}\int_{M}|u|^p\,dv_g+\frac{1}{q}\int_{M}|f^-||u|^q\,dv_g.
\end{equation}
We then consider the following two possible cases.

\textbf{Case 1}: Assume that
\begin{equation*}\label{e3.28}
\int_{M}|f^-||u|^{q}\,dv_g \geq \eta_0k\int_{M}|f^-|\,dv_g.
\end{equation*}
In this case, by \eqref{e3.27}, \eqref{e3.25} and the fact that $k\geq k_{1,q}$, the term $G_q$ can be estimated from below, that is,
\begin{equation}\label{e3.29}
\begin{split}
G_q(u) &\geq  \frac{h}{p}||u\|_p^p+\frac{\eta_0k}{q}\int_M |f^-|\,dv_g=-\frac{|h|}{p}||u\|_p^p+\frac{\eta_0k}{q}\int_M |f^-|\,dv_g\\
&\geq\frac{|h|}{p}k^{\frac{p}{q}}\Bigg(\underbrace{\frac{p \eta_0 k^{1-\frac{p}{q}}}{q|h|}\int_M |f^-|\,dv_g}_{\geq p}-1\Bigg)
\geq \frac{(p-1)|h|}{p}k^{\frac{p}{q}}.
\end{split}
\end{equation}

\textbf{Case 2}: Assume that
\begin{equation*}\label{e3.30}
\int_{M}|f^-||u|^{q}\,dv_g < \eta_0k\int_{M}|f^-|\,dv_g.
\end{equation*}
Then it is clear that $k^{-{1}/{q}}u\in \mathcal{A}_{\eta_0, q}'$ which implies
${\|\nabla_gu\|_p^p}\geq \lambda_{f,\eta_0, q}{\|u\|_p^{p}}$
by the definition of $\lambda_{f,\eta_0, q}$. Therefore, together with \eqref{e3.27}, the term $G_q$ can be estimated from below by
\begin{equation}\label{e3.32}
G_q(u) \geq  \delta\|u\|_p^p+\frac{1}{q}\int_{M}|f^-||u|^q\,dv_g,
\end{equation}
where $\delta$ is as in \eqref{e3.21}. Using \eqref{e3.27} we have
\begin{equation}\label{e3.33}
\|u\|_p^p=\frac{p}{|h|}\left(\frac{1}{p}\|\nabla_{g}u\|_p^{p}+\frac{1}{q}\int_{M}|f^-||u|^q\,dv_g-G_q(u)\right).
\end{equation}
Now we set $\delta\|u\|_p^p=\alpha\|u\|_p^p+\beta\|u\|_p^p$, where $\alpha=\frac{\beta A}{|h|\left(K(n,p)^p+1\right)}$,
$\delta=\alpha+\beta$.
We then apply \eqref{e3.32} and \eqref{e3.33} to get
\begin{equation*}\label{eq3.34}
\begin{split}
G_q(u)\geq &\alpha\|u\|_p^p+\frac{\beta p}{|h|}\left(\frac{1}{p}\|\nabla_{g}u\|_p^{p}+\frac{1}{q}\int_{M}|f^-||u|^q\,dv_g-G_q(u)\right)
+\frac{1}{q}\int_{M}|f^-||u|^q\,dv_g\\
\geq & \alpha\|u\|_p^p+\frac{\beta p}{|h|}\left(\frac{1}{p}\|\nabla_{g}u\|_p^{p}-G_q(u)\right),
\end{split}
\end{equation*}
which yields
\begin{equation}\label{e3.35}
G_q(u) \geq  \frac{\beta}{|h|+\beta p}\left(\|\nabla_{g}u\|_p^{p}+\frac{\alpha|h|}{\beta}\|u\|_p^p\right).
\end{equation}
By Sobolev's inequality, H\"{o}lder's inequality and the equality $\frac{A}{K(n,p)^p+1}=\frac{\alpha|h|}{\beta}$, we deduced that
\begin{equation*}\label{e3.36}
\|\nabla_{g}u\|_p^{p}+\frac{\alpha|h|}{\beta}\|u\|_p^p\geq \frac{k^{\frac{p}{q}}}{K(n,p)^p+1}.
\end{equation*}
Since $\beta=\frac{(K(n,p)^p+1)|h|\delta}{(K(n,p)^p+1)|h|+A}$, \eqref{e3.35} gives us
\begin{equation}\label{e3.37}
G_q(u) \geq  \frac{\beta}{|h|+p\beta}\cdot\frac{k^{\frac{p}{q}}}{K(n,p)^p+1}=\frac{\delta}{A+(K(n,p)^p+1)(|h|+p\delta)}k^{\frac{p}{q}}.
\end{equation}
It now follows from \eqref{e3.22}, \eqref{e3.29} and \eqref{e3.37} that $G_q(u)\geq mk^{\frac{p}{q}}$, where $m$ is as in \eqref{e3.22}.
Thus, we obtain
\begin{equation*}\label{e3.38}
\mathcal{I}_q^{\varepsilon}(u) \geq  mk^{\frac{p}{q}}-\frac{k}{q}\sup_Mf.
\end{equation*}
Thanks to \eqref{e3.23} and \eqref{e3.25}, we can choose $k_{1,q}\leq k<\left(\frac{mq}{2\sup_Mf}\right)^{\frac{q}{q-p}}$ such that
$\frac{1}{2}mk^{\frac{p}{q}}-\frac{k}{q}\sup_Mf>0$, and thus we get
$\mathcal{I}_q^{\varepsilon}(u) \geq  \frac{1}{2}mk^{\frac{p}{q}}>0$.
From \eqref{e3.22} and \eqref{e3.23}, we have
\begin{equation}\label{e3.39}
\sup_Mf\leq \mathcal{C}_q \int_M|f^-|\,dv_g=\frac{\eta_0m}{4|h|}\int_M|f^-|\,dv_g.
\end{equation}
It then follows from \eqref{e3.25} and \eqref{e3.39} that
\begin{equation*}
\left(\frac{mq}{2\sup_Mf}\right)^{\frac{q}{q-p}}\geq \left(\frac{2q|h|}{\eta_0\int_M|f^-|\,dv_g}\right)^{\frac{q}{q-p}}=2^{\frac{q}{q-p}}k_{1,q}>2^{\frac{n}{p}}k_{1,q}.
\end{equation*}
Hence, if we set $k_{2,q}=2^{\frac{n}{p}}k_{1,q}$, then for any $k\in [k_{1,q}, k_{2,q}]$ we always have
$\mathcal{I}_q^{\varepsilon}(u) \geq  \frac{1}{2}mk^{\frac{p}{q}}>0$. In other words, $\mu_{k,q}^{\varepsilon}>0$ for any
$k\in [k_{1,q}, k_{2,q}]$ and any $\varepsilon>0$. This completes the proof.
\end{proof}

\subsection{The Palais-Smale condition}
In this subsection, we will prove the Palais-Smale compact condition.
\begin{proposition}\label{prop3.3}
Suppose that the condition \eqref{e3.21}-\eqref{e3.23} hold. Then for each $\varepsilon>0$ fixed, the functional $\mathcal{I}_q^{\varepsilon}(\cdot)$
satisfies the Palais-Smale condition.
\end{proposition}
\begin{proof}
Let $\varepsilon>0$ be fixed. Take $c\in \mathbb{R}$ and assume that $\{v_j\}_j\subset H_1^p(M)$ is a Palais-Smale sequence at level $c$,
namely, such that
\begin{equation*}\label{ps1}
\mathcal{I}_q^{\varepsilon}(v_j)\to c\quad \text{and}\quad \delta \mathcal{I}_q^{\varepsilon}(v_j)\to 0 \quad \text{as}\quad j \to +\infty.
\end{equation*}
As the first step, we prove that, up to subsequences, $\{v_j\}_j$ is bounded in $H_1^p(M)$. Without loss of generality, we may
assume that $\|v_j\|\geq1$ for all $j$.
By means of the Palais-Smale condition, one can derive
\begin{equation}\label{e3.40}
\frac{1}{p}\|\nabla_{g} v_j\|_p^{p}+\frac{h}{p}\|v_j\|_p^p-\frac{1}{q}\int_{M}f|v_j|^q\,dv_g+\frac{1}{q}\int_{M}\frac{a}{(v_j^{2}+\varepsilon)^{\frac{q}{2}}}\,dv_g=c+o(1)
\end{equation}
and
\begin{equation}\label{e3.41}
\begin{split}
&\int_M |\nabla_gv_j|_g^{p-2}\left\langle\nabla_gv_j,\nabla_g\varphi\right\rangle_g\,dv_g+h\int_{M}|v_j|^{p-2}v_j\varphi\,dv_g\\
&\quad\quad\quad\quad\quad-\int_{M}f|v_j|^{q-2}v_j\varphi\,dv_g-\int_{M}\frac{av_j\varphi}{(v_j^{2}+\varepsilon)^{\frac{q}{2}+1}}\,dv_g=o(1)\|\varphi\|
\end{split}
\end{equation}
for any $\varphi\in H_1^p(M)$. Letting $\varphi=v_j$ in \eqref{e3.41}, we get
\begin{equation}\label{e3.42}
\|\nabla_{g}v_j\|_p^{p}+h\|v_j\|_p^p-\int_{M}f|v_j|^q\,dv_g-\int_{M}\frac{av_j^2}{(v_j^{2}+\varepsilon)^{\frac{q}{2}+1}}\,dv_g=o(1)\|\varphi\|.
\end{equation}
For simplicity, we denote $k_j=\int_M|v_j|^q\,dv_g$. There are two possible cases according to Proposition \ref{prop3.2}.

\textbf{Case 1}: Assume that there exists a subsequence of $\{v_j\}_j$, still denoted by $\{v_j\}_j$, such that
\begin{equation*}\label{e3.43}
\int_{M}|f^-||v_j|^{q}\,dv_g \geq \eta_0k_j\int_{M}|f^-|\,dv_g.
\end{equation*}
Using \eqref{e3.22} and \eqref{e3.23}, we get that
\begin{equation*}\label{e3.44}
\begin{split}
\mathcal{I}_q^{\varepsilon}(v_j) &\geq  \frac{h}{p}k_j^{\frac{p}{q}}+\frac{\eta_0k_j}{q}\int_{M}|f^-|\,dv_g-\frac{1}{q}\int_{M}f^+|v_j|^q\,dv_g\\
& \geq  \frac{h}{p}k_j^{\frac{p}{q}}+\frac{\eta_0k_j}{q}\int_{M}|f^-|\,dv_g-\frac{k_j}{q}\sup_Mf\\
& \geq  \frac{h}{p}k_j^{\frac{p}{q}}+\frac{\eta_0k_j}{q}\int_{M}|f^-|\,dv_g-\frac{k_j}{q}\frac{(p-1)\eta_0}{4p}\int_{M}|f^-|\,dv_g\\
& = \left(\frac{(3p+1)\eta_0}{4pq}\int_{M}|f^-|\,dv_g\right)k_j-\frac{|h|}{p}k_j^{\frac{p}{q}}.
\end{split}
\end{equation*}
This estimate and the fact that $0<\frac{p}{q}<1$ and $\mathcal{I}_q^{\varepsilon}(v_j)\rightarrow c$ imply that $\{k_j\}_j$ is bounded.
In other words, this means that $\{v_j\}_j$ is bounded in $L^q(M)$. Hence, H\"{o}lder's inequality and \eqref{e3.40} imply that $\{v_j\}_j$
is also bounded in $H_1^p(M)$.

\textbf{Case 2}: In contrast to Case 1, for all $j$ sufficiently large, we assume that
\begin{equation*}\label{e3.45}
\int_{M}|f^-||v_j|^{q}\,dv_g < \eta_0k_j\int_{M}|f^-|\,dv_g.
\end{equation*}
Applying \eqref{e3.40} and \eqref{e3.42}, we get that
\begin{equation*}\label{e3.46}
\begin{split}
-\frac{1}{q}\int_M f|v_j|^q\,dv_g=&\frac{p}{(q-p)q}\int_{M}\frac{a}{(v_j^{2}+\varepsilon)^{\frac{q}{2}}}\,dv_g
+\frac{1}{q-p}\int_{M}\frac{av_j^2}{(v_j^{2}+\varepsilon)^{\frac{q}{2}+1}}\,dv_g\\
&-\frac{pc}{q-p}+o(1)+o(1)\|v_j\|.
\end{split}
\end{equation*}
Therefore, we may rewrite $\mathcal{I}_q^{\varepsilon}$ as follows
\begin{equation}\label{e3.47}
\mathcal{I}_q^{\varepsilon}(v_j)\geq\frac{1}{p}\|\nabla_{g}v_j\|_p^{p}
+\frac{h}{p}\|v_j\|_p^p-\frac{pc}{q-p}+o(1)+o(1)\|v_j\|+A_j,
\end{equation}
where
\begin{equation*}\label{e3.48}
A_j=\frac{1}{q-p}\left(\int_{M}\frac{a}{(v_j^{2}+\varepsilon)^{\frac{q}{2}}}\,dv_g
+\int_{M}\frac{av_j^2}{(v_j^{2}+\varepsilon)^{\frac{q}{2}+1}}\,dv_g\right).
\end{equation*}
Dividing \eqref{e3.47} by $\|v_j\|_p$ and using the equivalent norm to $\|v_j\|_{H_1^p}=\|\nabla_{g}v_j\|_p+\|v_j\|_p$,
we obtain
\begin{equation}\label{e3.49}
\begin{split}
\frac{\mathcal{I}_q^{\varepsilon}(v_j)}{\|v_j\|_p}\geq &\frac{1}{p}\left(\frac{\|\nabla_{g}v_j\|_p^p}{\|v_j\|_p^p}+h\right)\|v_j\|_p^{p-1}
+o(1)\frac{\|\nabla_{g}v_j\|_p}{\|v_j\|_p}\\
&-\frac{pc}{(q-p)\|v_j\|_p}+o(1)+\frac{o(1)}{\|v_j\|_p}+\frac{A_j}{\|v_j\|_p}.
\end{split}
\end{equation}
Since $k_j^{-1/q}v_j\in \mathcal{A}_{\eta_0,q}'$, from the definition of $\lambda_{f,\eta_0,q}$, we have $\|\nabla_{g}v_j\|_p^p\geq \lambda_{f,\eta_0,q}\|v_j\|_p^p$.
Therefore, from \eqref{e3.49} and for $j$ large enough, one gets
\begin{equation}\label{e3.50}
\begin{split}
\frac{\mathcal{I}_q^{\varepsilon}(v_j)}{\|v_j\|_p}\geq&\frac{\lambda_{f,\eta_0,q}+h}{p}\|v_j\|_p^{p-1}+o(1)\lambda_{f,\eta_0,q}^{\frac{1}{p}}\\
&-\frac{pc}{(q-p)\|v_j\|_p}+o(1)+\frac{o(1)}{\|v_j\|_p}+\frac{A_j}{\|v_j\|_p}.
\end{split}
\end{equation}
If $\|v_j\|_p\rightarrow +\infty$ as $j\rightarrow +\infty$, then we clearly reach a contradiction by taking the limit
in \eqref{e3.50} since $\lambda_{f,\eta_0,q}+h>0$ and $A_j>0$, but
\begin{equation*}\label{e3.51}
\frac{\mathcal{I}_q^{\varepsilon}(v_j)}{\|v_j\|_p}\rightarrow 0 \quad \text{as } j\rightarrow +\infty.
\end{equation*}
Thus, $\{v_j\}_j$ is bounded in $L^p(M)$, which also implies that $\{\nabla_gv_j\}_j$ is bounded in $L^p(M)$ because of \eqref{e3.47}.
Consequently, $\{v_j\}_j$ is bounded in $H_1^p(M)$.
Combining Case 1 and Case 2 together, we conclude that there exists a bounded subsequence of $\{v_j\}_j$ in $H_1^p(M)$,
still denoted by $\{v_j\}_j$. This completes the first step.

We now prove that $v_j\rightarrow v$ strongly in $H_1^p(M)$ for some $v\in H_1^p(M)$.
Since $\{v_j\}_j$ is bounded in $H_1^p(M)$, there exists $v\in H_1^p(M)$ such that, up to subsequences,
\begin{align*}
\begin{split}
&v_j \rightharpoonup v \,\, \text{weakly in } H_1^p(M),\,\,v_j \rightarrow v \,\, \text{strongly in } L^{r}(M) \,\, \text{for all } r\in [1,p^*),\\
&\text{and }v_j(x) \rightarrow v(x) \,\, \text{for almost every } x\in M \,\,\text{as } j \to +\infty.
\end{split}
\end{align*}
From \eqref{e3.41} and the above convergence properties, we obtain
\begin{align*}
\begin{split}
&\int_{M}\left\langle|\nabla_gv_j|_g^{p-2}\nabla_gv_j-|\nabla_gv_i|_g^{p-2}\nabla_gv_i, \nabla_gv_j-\nabla_gv_i\right\rangle_g\,dv_g\\
&=|h|\int_{M}(|v_j|^{p-2}v_j-|v_i|^{p-2}v_i)(v_j-v_i)\,dv_g+\int_{M}f(|v_j|^{q-2}v_j-|v_i|^{q-2}v_i)(v_j-v_i)\,dv_g\\
&\quad\quad\quad+\int_{M}a\left(\frac{v_j}{(v_j^{2}+\varepsilon)^{\frac{q}{2}+1}}-\frac{v_i}{(v_i^{2}+\varepsilon)^{\frac{q}{2}+1}}\right)(v_j-v_i)\,dv_g+o(1)\|v_j-v_i\|\\
&=o(1).
\end{split}
\end{align*}
On the other hand, by Lindqvist's formula \cite[Page 162]{Lind}, for any $X, Y\in \mathbb{R}^n$, there exists a constant $c_p>0$ such that
\begin{equation*}\label{e3.54}
\langle |X|^{p-2}X-|Y|^{p-2}Y, X-Y \rangle \geq \begin{cases}
 c_p|X-Y|^p,\, &  \text{if } p\geq 2, \\
c_p\frac{|X-Y|^2}{(|X|+|Y|)^{2-p}}, &  \text{if } 1<p\leq 2,
\end{cases}
\end{equation*}
where $\langle \cdot, \cdot\rangle$ denotes the standard scalar product in $\mathbb{R}^n$.
Hence, if $p\geq2$, one gets that
\begin{equation}\label{e3.55}
c_p\|\nabla_gv_j-\nabla_gv_i\|_p^p
\leq\int_{M}\left\langle|\nabla_gv_j|_g^{p-2}\nabla_gv_j-|\nabla_gv_i|_g^{p-2}\nabla_gv_i, \nabla_gv_j-\nabla_gv_i\right\rangle_g\,dv_g.
\end{equation}
And if $1<p<2$, we obtain by H\"{o}lder's inequality that
\begin{equation}\label{e3.56}
\begin{split}
c_p\|\nabla_gv_j-\nabla_gv_i\|_p^p&\leq  c_p\left[\int_M\frac{|\nabla_g(v_j-v_i)|_g^2}{(|\nabla_gv_j|_g+|\nabla_gv_i|_g)^{2-p}}\right]^{\frac{p}{2}}
\left[\int_M(|\nabla_gv_j|_g+|\nabla_gv_i|_g)^p\right]^{\frac{2-p}{2}}\\
& \leq  C\left[\int_{M}\left\langle|\nabla_gv_j|_g^{p-2}\nabla_gv_j-|\nabla_gv_i|_g^{p-2}\nabla_gv_i, \nabla_gv_j-\nabla_gv_i\right\rangle_g\,dv_g\right]^{\frac{p}{2}}.
\end{split}
\end{equation}
Combining \eqref{e3.55} and \eqref{e3.56}, we then infer that
\begin{equation*}
\|\nabla_gv_j-\nabla_gv_i\|_p\rightarrow 0 \quad \text{as } i, j\rightarrow +\infty.
\end{equation*}
This shows that $\{v_j\}_j$ is a Cauchy sequence which, together with the completeness of $H_1^p(M)$, proves that
$\|v_j-v\|\rightarrow 0$ as $j\rightarrow +\infty$.
We then get the assertion.
\end{proof}

\section{Proof of Theorem \ref{theorem1}}

In this section, we prove Theorem \ref{theorem1}. This can be done through three steps.
\subsection{The existence of the first solution}

In this subsection, we obtain the existence of the first solution of \eqref{eq:M}. Notice that, we require \eqref{e4.1} to be hold. This restriction will be
removed by using a scaling argument in the last step.

\begin{proposition}\label{prop4.1}
Let $(M, g)$ be a smooth compact smooth Riemannian manifold without boundary of dimension $n$ $(n \geq 3)$. Let $h<0$ be a constant,
$f$ and $a\geq 0$ be smooth functions on $M$ with $\int_Ma\,dv_g>0$, $\int_Mf\,dv_g<0$, and $\sup_Mf>0$.
We further assume that
\begin{equation}\label{e4.1}
|h|\leq \frac{\eta_0}{p^*}\int_{M}|f^-|dv_g
\end{equation}
and
\begin{equation} \label{e4.2}
\int_{M}a\,dv_g< \frac{p}{2(n-p)}\Big(\frac{2n-p}{2(n-p)}\Big)^{\frac{2n}{p}-1}\Big(\frac{|h|}{\int_{M}|f^-|dv_g}\Big)^{\frac{2n}{p}}\int_M|f^-|\,dv_g
\end{equation}
hold, where $\eta_0$ is as in Proposition \ref{prop3.2}. Then there exists a number $\mathcal{C}_1$ given by \eqref{e4.5} below such that if
\begin{equation} \label{e4.3}
\sup_Mf\left(\int_{M}|f^-|\,dv_g\right)^{-1}\leq \mathcal{C}_1,
\end{equation}
then Eq. \eqref{eq:M} possesses at least two positive solutions.
\end{proposition}
\begin{proof} We divide the proof into several claims for the sake of clarity.

\begin{claim}\label{claim1}
\textit{There exists $q_0\in (p^{\flat}, p^*)$ such that for all $q\in (q_0, p^*)$ and $\varepsilon>0$ small enough,
there will be $k_0$ and $k_{*}$ with the following properties: $k_0>k_*$ and $\mu_{k_0, q}^{\varepsilon}\leq 0$ while $\mu_{k_*, q}^{\varepsilon}> 0$.}
\end{claim}

\noindent\textup{\textbf{Proof of Claim \ref{claim1}}}.
By Remark \ref{rm3.3} and \eqref{e4.2} there is some $q_0\in (p^{\flat}, p^*)$ such that the condition \eqref{e3.5} holds for all $q\in (q_0, p^*)$.
Hence, by Lemma \ref{asymptotic1} (iii), there exists $k_0>0$ such that $\mu_{k_0, q}^{\varepsilon}\leq 0$. Notice that
$p^{\flat}>p$ for any $n\geq 3$. From Lemma \ref{asymptotic1} (i), we can choose $k_*<\min\{k_0, 1\}$ such that $\mu_{k_*, q}^{\varepsilon}> 0$ for any $\varepsilon\leq k_*$.
Thus Claim 1 is proved.

\begin{claim}\label{claim2} \textit{ The following subcritical equation
\begin{equation}\label{e4.4}
\Delta_{p,g}u+hu^{p-1}=f(x)u^{q-1}+{a(x)}{u^{-q-1}},
\end{equation}
where $q\in (p^{\flat}, p^*)$, admits two positive solutions, say $u_{1,q}$ and $u_{2,q}$. Note that Eq. \eqref{e4.4}
corresponds to the case $\varepsilon=0$ in Eq. \eqref{eq:M2}.}
\end{claim}

\noindent\textup{\textbf{Proof of Claim \ref{claim2}.}}
By Proposition \ref{prop3.2}, we have $\eta_0$ and its corresponding $q_{\eta_0}\in (p^{\flat}, p^*)$ such that
$\delta=\frac{\lambda_{f, \eta_0, q}+h}{p}>\frac{1}{2p}(\lambda_{f}+h)$ for any $q\in (q_{\eta_0}, p^*)$.
By Lemma \ref{lemma3.4}, we have that $\frac{1}{2p}(\lambda_{f}+h)<\delta\leq \frac{1}{p}(\lambda_{f}+h)$.
It then follows immediately from this and \eqref{e3.21} that $\mathcal{C}_q\geq \mathcal{C}_1$, where
\begin{equation}\label{e4.5}
\mathcal{C}_1=\frac{\eta_0}{4|h|}\min\left\{\frac{\lambda_f+h}{2p\left[A+\left(K(n,p)^p+1\right)\lambda_f\right]}, \frac{(p-1)|h|}{p}\right\}
\end{equation}
and $\mathcal{C}_q$  is as in \eqref{e3.22}. Thus the condition \eqref{e3.23} holds according to \eqref{e4.3}. Note that $\mathcal{C}_1$ is independent of $q$ and thus never vanishing for any $q\in (q_{\eta_0}, p^*)$.
Using Proposition \ref{prop3.2} again, there exists an interval $I_q=[k_{1,q}, k_{2,q}]$ such that $\mu_{k,q}^{\varepsilon}>0$
for any $k\in I_q$. We then apply Lemma \ref{asymptotic1} (iii) and we conclude that $k_*<k_0<k_{1,q}$, where $k_*$ is given as in Claim \ref{claim1}.
Observe that
\begin{equation*}
\lim_{q\rightarrow p^*}k_{1,q}=\left(\frac{|h|p^*}{\eta_0\int_M|f^-|\, dv_g}\right)^{\frac{n}{p}}:=l \quad\text{and}
\quad \lim_{q\rightarrow p^*}k_{2,q}=2^{\frac{n}{p}}l.
\end{equation*}

With these information in hand, we proceed the proof of Claim 2 in four steps.
\begin{spacing}{1.5}
\noindent\textbf{Step 1.} \textit{The existence of solution $u_{1,q}^{\varepsilon}$ with strictly negative energy $\mathcal{E}_{1,q}^{\varepsilon}$ for \eqref{eq:M2}}.
\end{spacing}
We define
\begin{equation*}
\mathcal{E}_{1,q}^{\varepsilon}=\inf_{u\in \mathcal{D}_{k,q}}\mathcal{I}_q^{\varepsilon}(u),
\end{equation*}
where $\mathcal{D}_{k,q}=\{u\in H_1^p(M): \|u\|_q^q=k \,\,\text{and } k_*\leq k\leq k_{1,q}\}$.
Due to \eqref{e4.1}, one can check that $k_{1,q}$ is strictly monotone increasing with respect to $q$,
and thus $\|u\|_q^q<l$ for all $u\in \mathcal{D}_{k,q}$. It follows from Section \ref{sec3} that
$\mathcal{E}_{1,q}^{\varepsilon}$ is finite and non-positive.
Then we have, with the same argument as in Lemma \ref{lemma 3.2}, $\mathcal{E}_{1,q}^{\varepsilon}$ is achieved by
some positive function $u_{1,q}^{\varepsilon}$. In particular, $\mathcal{E}_{1,q}^{\varepsilon}$ is the energy
of $u_{1,q}^{\varepsilon}$. From Remark \ref{rm3.4} we know that the energy $\mathcal{E}_{1,q}^{\varepsilon}$ is strictly negative.
Obviously, $u_{1,q}^{\varepsilon}$ is a solution of \eqref{eq:M2}.
By the Ekeland Variational Principle and the Palais-Smale compact condition, there exists a bounded minimizing sequence in $H_1^p(M)$ for
$\mathcal{E}_{1,q}^{\varepsilon}$. Now the lower semi-continuity of $H_1^p$-norm implies that $\|u_{1,q}^{\varepsilon}\|$ is bounded with
the bound independent of $q$ and $\varepsilon$. If we set $\|u_{1,q}^{\varepsilon}\|_q^q=k_1^{\varepsilon}$
we then also have $k_1^{\varepsilon}\in (k_*,k_{1,q})$.
\begin{spacing}{1.5}
\noindent\textbf{Step 2.} \textit{The existence of solution $u_{1,q}$ with strictly negative energy $\mu_{k_1,q}$ for \eqref{e4.4}}.
\end{spacing}
Let $\{\varepsilon_j\}_j$ be a sequence of positive numbers such that $\varepsilon_j\rightarrow0$
as $j\rightarrow +\infty$. For each $j$, let $u_{1,q}^{\varepsilon_j}$ be a positive function satisfying
the following subcritical equation
\begin{equation}\label{e4.8}
\Delta_{p,g}u_{1,q}^{\varepsilon_j}+h(u_{1,q}^{\varepsilon_j})^{p-1}=f(x)(u_{1,q}^{\varepsilon_j})^{q-1}
+\frac{a(x)u_{1,q}^{\varepsilon_j}}{\big((u_{1,q}^{\varepsilon_j})^{2}+\varepsilon_j\big)^{\frac{q}{2}+1}}
\end{equation}
in $M$.
Since the sequence $\{u_{1,q}^{\varepsilon_j}\}_j$ is bounded in $H_1^p(M)$, up to a subsequence, we can assume
\begin{equation}\label{e4.9}
\begin{cases}
\,u_{1,q}^{\varepsilon_j} \rightharpoonup u_{1,q}\quad  \text{weakly in }H_1^p(M), \\
\,u_{1,q}^{\varepsilon_j} \rightarrow  u_{1,q}\quad \text{strongly in } L^{r}(M) \,\, \text{for all } r\in [1,p^*),\\
\,u_{1,q}^{\varepsilon_j} \rightarrow u_{1,q} \quad \text{for almost every } x\in M,
\end{cases}
\end{equation}
for some $u_{1,q}\in H_1^p(M)$ as $j\rightarrow +\infty$.
Moreover, as $\{|\nabla_g u_{1,q}^{\varepsilon_j}|_g\}_j$ is bounded in $L^p(M)$, one can also assume that
\begin{equation*}
|\nabla_g u_{1,q}^{\varepsilon_j}|_g^{p-2}\nabla_g u_{1,q}^{\varepsilon_j}\rightharpoonup \Sigma_{1,q}\,\, \text{weakly in } (L^p(M))^{\prime}.
\end{equation*}
By Lemma \ref{lemma 2.2} and Lebesgue's dominated convergence theorem, we have that
$\int_M(u_{1,q})^{-r}\,dv_g$ is finite for all $r$. Taking $j\rightarrow +\infty$ in \eqref{e4.8}, one then
gets that
\begin{equation*}\label{e4.10}
-\operatorname{div}_g(\Sigma_{1,q})+h(u_{1,q})^{p-1}=f(x)(u_{1,q})^{q-1}+a(x)(u_{1,q})^{-q-1}.
\end{equation*}
Since
\begin{equation*}
-h(u_{1,q})^{p-1}+f(x)(u_{1,q})^{q-1}+a(x)(u_{1,q})^{-q-1}
\end{equation*}
is bounded in $L^1(M)$, one can prove that $\Sigma_{1,q}=|\nabla_g u_{1,q}|_g^{p-2}\nabla_g u_{1,q}$, as in Demengel-Hebey \cite{Demengel98} (see also Chen-Liu \cite{ChenLiu}). Hence, $u_{1,q}$ is a weak solution of Eq. \eqref{e4.4}.
By Lemma \ref{lemma 2.1}, $u_{1,q}\in C^{1,\alpha}(M)$ for some $0<\alpha<1$.
Let $\|u_{1,q}\|_q^q=k_1$, by \eqref{e4.9}, we still have $k_1\in(k_*,k_{1,q})$. Consequently, it holds $u_{1,q}\not\equiv0$.
By Lemma \ref{lemma 2.1} and strong maximum principle, we get $u_{1,q}>0$ in $M$.
Since $u_{1,q}^{\varepsilon_j}$ has strictly negative energy
$\mathcal{E}_{1,q}^{\varepsilon_j}$, passing to the limit as $j\rightarrow +\infty$, we have that
$u_{1,q}$ also has strictly negative energy, which we denoted by $\mu_{k_1,q}$. Thus, we already show that $u_{1,q}$ is a positive solution of
\eqref{e4.4} as claimed.

\begin{spacing}{1.5}
\noindent\textbf{Step 3.} \textit{The existence of solution $u_{2,q}^{\varepsilon}$ with strictly positive energy $\mathcal{E}_{2,q}^{\varepsilon}$ for \eqref{eq:M2}}.
\end{spacing}

Let $k^*$ be a real number such that
\begin{equation*}\label{e4.11}
\mu_{k^{*},q}^{\varepsilon}=\max\{\mu_{k,q}^{\varepsilon}: k_{1,q}\leq k\leq k_{2,q}\}.
\end{equation*}
From Proposition \ref{prop3.2}, we have $\mu_{k^{*},q}^{\varepsilon}>0$. According to Proposition \ref{prop3.1}, one can choose $\bar{k}_1\in(k_0,k_{1,q})$ and $\bar{k}_2\in(k_{2,q},k_{**})$ such that $\mu_{\bar{k}_1,q}^{\varepsilon}=\mu_{\bar{k}_2,q}^{\varepsilon}=0$. We have proved that
$\mu_{\bar{k}_1,q}^{\varepsilon}$ and $\mu_{\bar{k}_2,q}^{\varepsilon}$ can be achieved, say by
$u_{\bar{k}_1,q}$ and $u_{\bar{k}_2,q}$ respectively. We now set
\begin{equation*}\label{e4.12}
\Gamma=\{\gamma\in C([0,1];H_1^p(M)):\gamma(0)=u_{\bar{k}_1,q}, \gamma(1)=u_{\bar{k}_2,q}\}.
\end{equation*}
Consider the functional $E(v)=\mathcal{I}_q^{\varepsilon}(u_{\bar{k}_1,q}+v)$ for any non-negative real valued function $v$
with
\newcommand{\normmm}[1]{{\left\vert\kern-0.25ex\left\vert\kern-0.25ex\left\vert #1
    \right\vert\kern-0.25ex\right\vert\kern-0.25ex\right\vert}}
\begin{equation*}\label{e4.14}
\normmm v=\left(\int_M\big|u_{\bar{k}_1,q}+v\big|^q\,dv_g\right)^{\frac{1}{q}}.
\end{equation*}
Clearly, we have $E(0)=0$. Denote $\rho=(k^*)^{{1}/{q}}$. If $\normmm v=\rho$, by setting $u=u_{\bar{k}_1,q}+v$, then $\|u\|_q^q=k^*$.
Therefore
\begin{equation*}\label{e4.15}
E(v)=\mathcal{I}_q^{\varepsilon}(u)\geq \mu_{k^{*},q}^{\varepsilon}>0.
\end{equation*}
Next we set $v_1=u_{\bar{k}_2,q}-u_{\bar{k}_1,q}$, and then obtain
$$
E(v_1)=\mathcal{I}_q^{\varepsilon}(u_{\bar{k}_2,q})=0\,\,\, \text{and}\,\,\,
\normmm {v_1}=\left(\int_M\big|u_{\bar{k}_2,q}\big|^q\,dv_g\right)^{\frac{1}{q}}=(\bar{k}_2)^{\frac{1}{q}}>\rho.
$$
Notice that the functional $E$ satisfies the Palais-Smale condition as we have shown for $\mathcal{I}_q^{\varepsilon}$.
Thus, by Theorem 6.1 in \cite[Chapter II]{MS}, the following mountain pass level
$$
\mathcal{E}_{2,q}^{\varepsilon}:= \inf_{\gamma \in \Gamma} \sup_{t \in [0,1]} E(\gamma(t)-u_{\bar{k}_1,q})
$$
is a critical value of the functional $E$.  It is then obvious that $\mathcal{E}_{2,q}^{\varepsilon}>0$.
Thus, there exists a Palais-Smale sequence $\{u_j\}_j\subset H_1^p(M)$ for $\mathcal{I}_q^{\varepsilon}$ at the
level $\mathcal{E}_{2,q}^{\varepsilon}$. Since $\mathcal{I}_q^{\varepsilon}(u_j)=\mathcal{I}_q^{\varepsilon}(|u_j|)$ for every $j$, we
can assume $u_j\geq 0$ for any $j$. Consequently, Proposition \ref{prop3.3} implies that, up to subsequences,
$u_j\rightarrow u_{2,q}^{\varepsilon}$ strongly in $H_1^p(M)$ for some $u_{2,q}^{\varepsilon}\in H_1^p(M)$ as $j\rightarrow +\infty$.
Therefore, the function $u_{2,q}^{\varepsilon}$ with positive energy $\mathcal{E}_{2,q}^{\varepsilon}$ satisfies
the following subcritical equation
\begin{equation}\label{e4.16}
\Delta_{p,g}u_{2,q}^{\varepsilon}+h(x)(u_{2,q}^{\varepsilon})^{p-1}=f(x)(u_{2,q}^{\varepsilon})^{q-1}
+\frac{a(x)u_{2,q}^{\varepsilon}}{\big((u_{2,q}^{\varepsilon})^{2}+\varepsilon\big)^{\frac{q}{2}+1}}
\end{equation}
in the weak sense.
The non-negativity of $\{u_j\}_j$ implies that $u_{2,q}^{\varepsilon}\geq 0$ almost everywhere, and thus the regularity result, i.e., Lemma \ref{lemma 2.1}(i)
can be applied to \eqref{e4.16}. It follows that  $u_{2,q}^{\varepsilon}\in C^{1, \alpha}(M)$ for some $\alpha \in (0, 1)$, which
also implies $u_{2,q}^{\varepsilon}\geq 0$ in $M$.
 We claim that $u_{2,q}^{\varepsilon}\not\equiv 0$. Indeed, by Lemma \ref{asymptotic1} (iv) we have that $\mathcal{E}_{2,q}^{\varepsilon}\leq \mu<\infty$.
If $u_{2,q}^{\varepsilon}\equiv 0$, then we have
\begin{equation*}
\frac{1}{q\varepsilon^{\frac{q}{2}}}\int_Ma\,dv_g=\mathcal{E}_{2,q}^{\varepsilon}\leq \mu<\infty,
\end{equation*}
which is impossible if $\varepsilon\rightarrow 0$. Thus, $u_{2,q}^{\varepsilon}>0$ on $M$ if $\varepsilon$ is sufficiently small which
we will always assume from now on. Denote $\|u_{2,q}^{\varepsilon}\|_q^q=k_2^{\varepsilon}$. In view of Lemma \ref{asymptotic1} (v), we know that $k_2^{\varepsilon}>0$ is bounded from above by $k_{**}$
independent of both $\varepsilon$ and $q$.
\begin{spacing}{1.5}
\noindent\textbf{Step 4.} \textit{The existence of solution $u_{2,q}$ with strictly positive energy $\mu_{k_2,q}$ for \eqref{e4.4}}.
\end{spacing}
Let $\{\varepsilon_j\}_j$ be a sequence of positive numbers such that $\varepsilon_j\rightarrow0$
as $j\rightarrow +\infty$. For each $j$, let $u_{2,q}^{\varepsilon_j}$ be a positive function in $M$ satisfies
the following subcritical equation
\begin{equation}\label{e4.18}
\Delta_{p,g}u_{2,q}^{\varepsilon_j}+h(x)(u_{2,q}^{\varepsilon_j})^{p-1}=f(x)(u_{2,q}^{\varepsilon_j})^{q-1}
+\frac{a(x)u_{2,q}^{\varepsilon_j}}{\big((u_{2,q}^{\varepsilon_j})^{2}+\varepsilon_j\big)^{\frac{q}{2}+1}}
\end{equation}
in $M$. By the discussion above, the sequence $\{u_{2,q}^{\varepsilon_j}\}_j$ is bounded in $H_1^p(M)$. Consequently, there exists $u_{2,q}\in H_1^p(M)$ such that up to a subsequence,
\begin{align*}
\begin{split}
&u_{2,q}^{\varepsilon_j} \rightharpoonup u_{2,q} \,\, \text{weakly in } H_1^p(M),\,\,u_{2,q}^{\varepsilon_j} \rightarrow  u_{2,q} \,\, \text{strongly in } L^{r}(M) \,\, \text{for all } r\in [1,p^*),\\
&\text{and }u_{2,q}^{\varepsilon_j}(x) \rightarrow  u_{2,q}(x) \,\, \text{for almost every } x\in M\,\, \text{as} \,\,j \to +\infty.
\end{split}
\end{align*}
Thus, $u_{2,q}\geq 0$ almost everywhere in $M$. We now denote $\|u_{2,q}\|_q^q=k_2$. By Lemma \ref{lemma 2.2}, the sequence
$\{u_{2,q}^{\varepsilon_j}\}_j$ is uniformly bounded from below. Applying Lebesgue's dominated convergence theorem, we have that $(u_{2,q})^{-1}\in L^r(M)$ for any $r>0$.
Taking $j\rightarrow +\infty$ in \eqref{e4.18} and deducing as in Step 2, we get that
$u_{2,q}$ is the second weak solution of the following subcritical equation
\begin{equation*}
\Delta_{p,g}u_{2,q}+h(x)(u_{2,q})^{p-1}=f(x)(u_{2,q})^{q-1}
+{a(x)}{(u_{2,q})^{-q-1}}.
\end{equation*}
It follows from Lemma \ref{lemma 2.1}(ii) that $u_{2,q}\in C^{1,\alpha}(M)$ for some $\alpha\in (0,1)$, and thus $u_{2,q}>0$ in $M$.
Since $u_{2,q}^{\varepsilon_j}$ has positive energy $\mathcal{E}_{2,q}^{\varepsilon_j}$,
taking the limit as $j\rightarrow +\infty$, we know that the energy of $u_{2,q}$ is still positive, i.e., $\mu_{k_2,q}>0$. Moreover, according to Step 2, it holds $u_{1,q}\not\equiv u_{2,q}$.
Note that $k_2$ is still bounded from above by $k_{**}$ independent of both $\varepsilon$ and $q$. This completes the proof of Claim 2.

\begin{claim}\label{claim3}
\textit{Eq. \eqref{eq:M} has at least one positive solution.}
\end{claim}
\noindent\textup{\textbf{Proof of Claim \ref{claim3}.}}
By Step 2 and Step 4 in Claim \ref{claim2} above, we know that there exist two positive functions $u_{1,q}$ and $u_{2,q}$ which solve \eqref{e4.4}.
Moreover, $\|u_{i,q}\|_q^q=k_i\,(i=1,2)$. We now estimate $\mu_{k_1,q}$ and $\mu_{k_2,q}$.
Recall that $\mu_{k_i,q}$ are the energy of $u_{i,q}$ found in Claim \ref{claim2}, i.e.,
\begin{equation*}\label{e4.20}
\mu_{k_i,q}=\frac{1}{p}\|\nabla_{g} u_{i,q}\|_p^{p}+\frac{h}{p}\|u_{i,q}\|_p^p
-\frac{1}{q}\int_{M}f(u_{i,q})^q\,dv_g+\frac{1}{q}\int_{M}\frac{a}{(u_{i,q})^{q}}\,dv_g.
\end{equation*}
We have already seen that $\mu_{k_1,q}<0<\mu_{k_2,q}\leq \mu$. Since $k_1\in (k_{*},k_{1,q})$ and $h<0$, we obtain
\begin{equation*}
\frac{1}{p}\|\nabla_{g} u_{1,q}\|_p^{p}\leq\mu_{k_1,q}-\frac{h}{p}k_1^{\frac{p}{q}}+\frac{1}{q}\int_{M}f(u_{1,q})^q\,dv_g
\leq\frac{k_1}{q}\sup_Mf-\frac{h}{p}k_1^{\frac{p}{q}},
\end{equation*}
and thus the sequence $\{u_{1,q}\}_q$ is bounded in $H_{1}^{p}(M)$. Similarly, from Lemma \ref{asymptotic1} (iv) and the following
estimate
\begin{equation*}
\frac{1}{p}\|\nabla_{g} u_{2,q}\|_p^{p}\leq\mu_{k_2,q}-\frac{h}{p}k_2^{\frac{p}{q}}+\frac{1}{q}\int_{M}f(u_{2,q})^q\,dv_g
\leq\mu+\frac{k_2}{q}\sup_Mf-\frac{h}{p}k_2^{\frac{p}{q}},
\end{equation*}
we know that the sequence $\{u_{2,q}\}_q$ is also bounded in $H_{1}^{p}(M)$.
Combining these facts, we get that for $i=1, 2$
\begin{equation*}
\|u_{i,q}\|^{p}=\|\nabla_{g}u_{i,q}\|_p^{p}+\|u_{i,q}\|_p^p\leq p\mu+\frac{pk_i}{q}\sup_Mf+(1-h)k_i^{\frac{p}{q}}.
\end{equation*}
Thanks to $k_{**}>1$, $k_i\leq k_{**}$ and $q>p^{\flat}$, if we denote
\begin{equation}\label{e4.24}
\Lambda=\Big[p\mu+(\sup_Mf)k_{**}+(1+|h|)k_{**}^{\frac{p}{p^{\flat}}}\Big]^{\frac{1}{p}},
\end{equation}
we then get that $\|u_{i,q}\|\leq \Lambda$ for $i=1, 2$. Thus, as usual, up to subsequences, there exists $u_i\in H_1^p(M)$  and $\Sigma_{i}\in L^{\frac{p}{p-1}}(M)$ such that,
as $q$ approaches to $p^*$,
\begin{equation*}
\begin{cases}
\,u_{i,q} \rightharpoonup u_{i}\quad  \text{weakly in }H_1^p(M), \\
\,u_{i,q} \rightarrow  u_{i}\quad \text{strongly in } L^{r}(M) \,\, \text{for all } r\in [1,p^*),\\
\,u_{i,q} \rightarrow u_{i} \quad \text{for almost every } x\in M,\\
\,|\nabla_g u_{i,q}|_g^{p-2}\nabla_g u_{i,q}\rightharpoonup \Sigma_{i}\,\, \text{weakly in } (L^p(M))^{\prime}.
\end{cases}
\end{equation*}
By Lemma \ref{lemma 2.2} and Lebesgue's dominated convergence theorem, it holds
\begin{equation*}
\lim_{q\rightarrow p^*}\int_M\frac{a(x)v}{(u_{i,q})^{q+1}}\,dv_g=\int_M\frac{a(x)v}{(u_{i})^{p^*+1}}\,dv_g\,\, \text{ for all }  v\in H_1^p(M).
\end{equation*}
Moreover, by \cite[Theorem 3.45]{Aubin1}, one has
\begin{equation}\label{e4.26}
(u_{i,q})^{q-1} \rightharpoonup (u_{i})^{p^*-1}\,\,  \text{weakly in }L^{\frac{p^*}{p^*-1}}(M)\,\, \text{as }  q\rightarrow p^*,
\end{equation}
since $\{(u_{i,q})^{q-1}\}_q$ is bounded in $L^{\frac{p^*}{q-1}}(M)\subset L^{\frac{p^*}{p^*-1}}(M)$ and $u_{i,q}\rightarrow u_i$ almost everywhere in $M$. We thus have, by Sobolev inequality and
the fact that $\big(L^{\frac{p^*}{p^*-1}}(M)\big)'=L^{p^*}(M)$,
\begin{equation*}
\lim_{q\rightarrow p^*}\int_M{f}{(u_{i,q})^{q-1}}v\,dv_g=\int_M{f}{(u_{i})^{p^*-1}}v\,dv_g \,\, \text{ for all }  v\in H_1^p(M).
\end{equation*}
Since $u_{i,q}$ satisfies the following identity
\begin{align}\label{e4.28}
\begin{split}
&\int_{M}|\nabla_gu_{i,q}|_g^{p-2}\left\langle\nabla_gu_{i,q}, \nabla_gv\right\rangle_g\,dv_g
+h\int_M(u_{i,q})^{p-1}v\,dv_g\\
&\quad\quad\quad\quad\quad\quad-\int_M{f}{(u_{i,q})^{q-1}}v\,dv_g-\int_M\frac{a(x)v}{(u_{i,q})^{q+1}}\,dv_g=0,
\end{split}
\end{align}
for all $ v\in H_1^p(M)$, passing to the limit for $q\rightarrow p^*$ in \eqref{e4.28}, we obtain for $i=1,2$ that
\begin{align*}
\begin{split}
&\int_{M}\left\langle \Sigma_i, \nabla_gv\right\rangle_g\,dv_g
+h\int_M(u_{i})^{p-1}v\,dv_g\\
&\quad\quad\quad\quad\quad\quad-\int_M{f}{(u_{i})^{p^*-1}}v\,dv_g-\int_M\frac{a(x)v}{(u_{i})^{p^*+1}}\,dv_g=0 \,\, \text{ for all }v\in H_1^p(M).
\end{split}
\end{align*}
With the similar argument as in Claim 2, we have $\Sigma_i=|\nabla_gu_{i}|_g^{p-2}\nabla_gu_{i}$.
This shows that $u_i$ ($i=1,2$) are weak solutions to \eqref{eq:M}.
By the strong maximum principle and Lemma \ref{lemma 2.1}(ii), one then gets that $u_i$ is positive and $u_i\in C^{1,\alpha}(M)$ for some $\alpha\in (0,1)$.
\end{proof}

\subsection{The existence of the second solution}

In the previous subsection, we only show that $u_i$ $(i=1,2)$ are solutions of \eqref{eq:M}. However, it is not clear whether
these functions are distinct. In this subsection, we will show that $u_i$ $(i=1,2)$ are in fact different provided $\sup_Mf$
is sufficiently small. Recall that the energies of $u_i$ are given as follows
\begin{equation*}
{\mathcal{I}}_{p^*}^{0}(u_i)=\frac{1}{p}\int_{M}|\nabla_{g} u_i|_g^{p}\,dv_g+\frac{h}{p}\int_{M}(u_i)^p\,dv_g
-\frac{1}{p^*}\int_{M}f(u_i)^{p^*}\,dv_g+\frac{1}{p^*}\int_{M}\frac{a}{(u_i)^{p^*}}\,dv_g.
\end{equation*}
As in \cite{Ngo1}, we need compare $\mathcal{I}_{p^*}^0(u_1)$ and $\mathcal{I}_{p^*}^0(u_2)$, and this could be done
once we can show that
$$
\lim_{q\rightarrow p^*}{\mathcal{I}}_{q}^{0}(u_{i,q})=\mathcal{I}_{p^*}^0(u_i)\,\, \text{ for } i=1, 2.
$$
From the expression of ${\mathcal{I}}_{q}^{0}(u_{i,q})$, we know that the only difficult part is to show that
\begin{equation}\label{e4.30}
\int_Mf(u_{i,q})^q\,dv_g\rightarrow\int_M{f}{(u_{i})^{p^*}}\,dv_g \,\,\, \text{as } \,q\rightarrow p^*.
\end{equation}
To this end, we make $\sup_Mf$ sufficiently small. Intuitively, such a small $f$ is equivalent to saying, for example,
that $f(u_{i,q})^q$ behaves exactly the same as $f(u_{i,q})^{q-1}$.
\begin{proposition}\label{prop4.2}
Assume that all requirements in \textup{Proposition \ref{prop4.1}} are fulfilled and $f$ satisfies
\begin{equation*}
\sup_Mf<\mathcal{C}_2,
\end{equation*}
where the number $\mathcal{C}_2$ is given in {\eqref{e4.39}} below. Then \eqref{e4.30} holds.
\end{proposition}
\begin{proof} In \eqref{e4.28}, we choose $v=(u_{i,q})^{1+p\delta}$ for some $\delta>0$ to be determined later.
Then we have
\begin{align}\label{e4.32}
\begin{split}
&\frac{1+p\delta}{(1+\delta)^p}\int_{M}|\nabla_gw_{i,q}|_g^{p}\,dv_g\\
&\quad\quad=|h|\int_M(w_{i,q})^{p}\,dv_g+\int_M{f}{(w_{i,q})^{p}}(u_{i,q})^{q-p}\,dv_g+\int_M\frac{a(x)}{(u_{i,q})^{q-p\delta}}\,dv_g,
\end{split}
\end{align}
where $w_{i,q}=(u_{i,q})^{\delta+1}$. Then, it follows from \eqref{e4.32} and Sobolev inequality that
\begin{align}\label{e4.33}
\begin{split}
\|w_{i,q}\|_{p^*}^p\leq& \bigg(\frac{(K(n,p)^p+1)(1+\delta)^p|h|}{1+p\delta}+A\bigg)\|w_{i,q}\|_{p}^p\\
&+\frac{(K(n,p)^p+1)(1+\delta)^p}{1+p\delta}\bigg(\int_Mf^+(w_{i,q})^p(u_{i,q})^{q-p}\,dv_g+\int_M\frac{a(x)}{(u_{i,q})^{q-p\delta}}\,dv_g\bigg).
\end{split}
\end{align}
By H\"{o}lder's inequality, we find
\begin{equation}\label{e4.34}
\int_M(w_{i,q})^p(u_{i,q})^{q-p}\,dv_g\leq \bigg(\int_M(w_{i,q})^{p^*}\,dv_g\bigg)^{\frac{p}{p^*}}\bigg(\int_M(u_{i,q})^{\frac{p^*(q-p)}{p^*-p}}\,dv_g\bigg)^{\frac{p^*-p}{p^*}}.
\end{equation}
Notice that $\frac{p^*(q-p)}{p^*-p}<q$ as long as $q<p^*$. Again by H\"{o}lder and Sobolev inequalities, one gets
\begin{equation}\label{e4.35}
\int_M(u_{i,q})^{\frac{p^*(q-p)}{p^*-p}}\,dv_g\leq \left(\int_M(u_{i,q})^{p^*}\,dv_g\right)^{\frac{q-p}{p^*-p}}
\leq \big(K(n,p)^p+1+A\big)^{\frac{p^*(q-p)}{p(p^*-p)}}\|u_{i,q}\|^{\frac{p^*(q-p)}{p^*-p}}.
\end{equation}
Therefore, from \eqref{e4.34} and \eqref{e4.35}, we have
\begin{equation}\label{e4.36}
\int_M(w_{i,q})^p(u_{i,q})^{q-p}\,dv_g\leq \|w_{i,q}\|_{p^*}^p \big(K(n,p)^p+1+A\big)^{\frac{q-p}{p}}\|u_{i,q}\|^{q-p}.
\end{equation}
Now, combing \eqref{e4.33} and \eqref{e4.36}, we easily get that
\begin{align}\label{e4.37}
\begin{split}
\|w_{i,q}\|_{p^*}^p\leq&\bigg(\frac{(K(n,p)^p+1)(1+\delta)^p|h|}{1+p\delta}+A\bigg)\|w_{i,q}\|_{p}^p\\
&+\underbrace{\frac{(K(n,p)^p+1)(1+\delta)^p}{1+p\delta}\big(K(n,p)^p+1+A\big)^{\frac{q-p}{p}}\big(\sup_Mf\big)\|u_{i,q}\|^{q-p}}_{\kappa}\|w_{i,q}\|_{p^*}^p\\
&+\frac{(K(n,p)^p+1)(1+\delta)^p}{1+p\delta}\int_M\frac{a(x)}{(u_{i,q})^{q-p\delta}}\,dv_g.
\end{split}
\end{align}
We wish to impose some condition of $\sup_Mf$ such that
\begin{equation}\label{e4.38}
\kappa\leq\frac{(K(n,p)^p+1)(1+\delta)^p}{1+p\delta}\big(K(n,p)^p+1+A\big)^{\frac{p^*-p}{p}}\big(\sup_Mf\big)\Lambda^{p^*-p}<\frac{1}{2},
\end{equation}
where $\Lambda>1$ is as in \eqref{e4.24}. To do so, we let $\sup_Mf<\mathcal{C}_2$ with
\begin{equation}\label{e4.39}
\mathcal{C}_2=\min\left\{\frac{1}{2\big(K(n,p)^p+1\big)}\big(K(n,p)^p+1+A\big)^{-\frac{p^*-p}{p}}\left(\mu p+k_{**}+(1+|h|)k_{**}^{\frac{p}{p^{\flat}}}\right)^{-\frac{p^*-p}{p}},1\right\}.
\end{equation}
Then, we must have
\begin{equation}\label{e4.40}
\big(K(n,p)^p+1\big)\big(K(n,p)^p+1+A\big)^{\frac{p^*-p}{p}}\big(\sup_Mf\big)\Lambda^{p^*-p}<\frac{1}{2}.
\end{equation}
Now, by \eqref{e4.40} and the fact that $\frac{(1+\delta)^p}{1+p\delta}\rightarrow 1$ as $\delta\rightarrow 0$, we can choose
a small $\delta>0$ such that \eqref{e4.38} holds. From now on, we fixed this $\delta$ in \eqref{e4.38} with $p(1+\delta)<p^*$ and $q-p\delta>0$.
In view of \eqref{e4.37}, it is easy to get
\begin{equation}\label{e4.41}
\|w_{i,q}\|_{p^*}^p\leq 2\bigg(\frac{(K(n,p)^p+1)(1+\delta)^p|h|}{1+p\delta}+A\bigg)\|w_{i,q}\|_{p}^p
+2\frac{(K(n,p)^p+1)(1+\delta)^p}{1+p\delta}\int_M\frac{a(x)}{(u_{i,q})^{q-p\delta}}\,dv_g.
\end{equation}
Moreover, by H\"{o}lder's inequality, we have
\begin{equation}\label{e4.42}
\|w_{i,q}\|_p=\|(u_{i,q})^{1+\delta}\|_p=\|u_{i,q}\|_{p(1+\delta)}^{1+\delta}\leq\|u_{i,q}\|_{p^*}^{1+\delta}.
\end{equation}
Then together with Sobolev inequality it yields that $\|w_{i,q}\|_p$ can be controlled by some constant depending on $\Lambda$.
By Lemma \ref{lemma 2.2} and the fact that $q-p\delta>0$, we know that $\int_Ma(u_{i,q})^{-(q-p\delta)}\,dv_g$
is bounded independently of $q$. Combining this with \eqref{e4.41}-\eqref{e4.42}, it shows that $\|w_{i,q}\|_{p^*}$ is bounded,
that is, $\|u_{i,q}\|_{p^*(1+\delta)}$ is bounded. Again with H\"{o}lder's inequality, we obtain
\begin{equation*}
\|(u_{i,q})^q\|_{1+\delta}\leq\|u_{i,q}\|_{p^*(1+\delta)}^{q},
\end{equation*}
which implies that $(u_{i,q})^q$ is bounded in $L^{1+\delta}(M)$.
Thus, by \cite[Theorem 3.45]{Aubin1} and $(u_{i,q})^q\rightarrow (u_i)^{p^*}$ almost everywhere in $M$ as $q\rightarrow p^*$, one gets that
$(u_{i,q})^{q} \rightharpoonup (u_{i})^{p^*}$ weakly in $L^{1+\delta}(M)$ as $q\rightarrow p^*$. Hence, by the definition of weak convergence
and the fact that $L^{1+1/{\delta}}(M)$ is the dual space of $L^{1+\delta}(M)$, there holds
\begin{equation*}
\int_Mf(u_{i,q})^q\,dv_g\rightarrow\int_M{f}{(u_{i})^{p^*}}\,dv_g
\end{equation*}
as $q\rightarrow p^*$, since we clearly have $f\in L^{1+1/{\delta}}(M)$.
\end{proof}

With the help of the Proposition \ref{prop4.2}, we can easily get the following result.
\begin{proposition}\label{prop4.3}
Assume that all requirements in \textup{Proposition \ref{prop4.2}} are fulfilled. Then there holds
$$\|\nabla_{g}u_{i,q}\|_p\rightarrow \|\nabla_{g}u_{i}\|_p \quad \text{as }\, q\rightarrow p^*.$$
\end{proposition}
\begin{proof}
It suffices to prove that $\nabla_{g}u_{i,q} \rightarrow \nabla_{g}u_{i}$ strongly in $L^p(M)$ as $q\rightarrow p^*$.
The choice of $u_{i,q}-u_i$  as a test function in \eqref{e4.28} gives us
\begin{align}\label{e4.43}
\begin{split}
&\int_{M}\left\langle|\nabla_gu_{i,q}|_g^{p-2}\nabla_gu_{i,q}-|\nabla_gu_{i}|_g^{p-2}\nabla_gu_{i}, \nabla_gu_{i,q}-\nabla_gu_{i}\right\rangle_g\,dv_g\\
&\quad\quad\quad=-h\int_M(u_{i,q})^{p-1}(u_{i,q}-u_{i})\,dv_g+\int_M{f}{(u_{i,q})^{q-1}}(u_{i,q}-u_{i})\,dv_g\\
&\quad\quad\quad\quad\,\,+\int_M\frac{a(x)}{(u_{i,q})^{q+1}}(u_{i,q}-u_{i})\,dv_g-\int_{M}|\nabla_gu_{i}|_g^{p-2}\left\langle\nabla_gu_{i}, \nabla_g(u_{i,q}-u_{i})\right\rangle_g\,dv_g.
\end{split}
\end{align}
We study the right-hand side of \eqref{e4.43}. It is straightforward to check with the usual arguments that, the limits of the first term, the third term and the forth term
vanish as $q\rightarrow p^*$. While for the second one, we have, using \eqref{e4.26} and Proposition \ref{prop4.2},
\begin{equation*}
\int_M{f}(u_{i,q})^{q}\,dv_g-\int_M{f}u_{i}(u_{i,q})^{q-1}\,dv_g\rightarrow 0  \quad \text{as }\, q\rightarrow p^*.
\end{equation*}
Therefore, we deduce
\begin{equation*}
\int_{M}\left\langle|\nabla_gu_{i,q}|_g^{p-2}\nabla_gu_{i,q}-|\nabla_gu_{i}|_g^{p-2}\nabla_gu_{i}, \nabla_gu_{i,q}-\nabla_gu_{i}\right\rangle_g\,dv_g\rightarrow0 \quad \text{as }\, q\rightarrow p^*.
\end{equation*}
Now,  by a similar argument as in Proposition \ref{prop3.3}, one gets
$\nabla_{g}u_{i,q} \rightarrow \nabla_{g}u_{i}$ strongly in $L^p(M)$ as $q\rightarrow p^*$.
\end{proof}

We can now easily conclude that Eq. \eqref{eq:M} has at least two positive solutions.
\begin{proposition}\label{prop4.4}
Assume that all requirements in \textup{Proposition \ref{prop4.2}} are fulfilled, then Eq. \eqref{eq:M} has at least two positive solutions, while
one has strictly negative energy and the other has positive energy.
\end{proposition}
\begin{proof}
It suffices to compare the energies of $u_i$ ($i=1,2$). By Proposition \ref{prop4.2} and \ref{prop4.3}, one has
$\lim_{q\rightarrow p^*}{\mathcal{I}}_{q}^{0}(u_{i,q})=\mathcal{I}_{p^*}^0(u_i)$ for $i=1, 2$.
According to \eqref{e3.6.2}, we have $\mathcal{I}_{p^*}^0(u_1)<0<\mathcal{I}_{p^*}^0(u_2)$. Thus, $u_1$ and $u_2$ have different energies.
This completes the proof.
\end{proof}

\subsection{The scaling argument}

In this subsection, we use the scaling technique to complete the proof of Theorem \ref{theorem1} by removing the condition
\eqref{e4.1} mentioned in Proposition \ref{prop4.1}. We first observe that under the variable change $\widetilde{u}=\frac{u}{c}$,
where $c$ is a suitable constant to be determined later, Eq. \eqref{eq:M} becomes
\begin{equation}\label{e4.44}
\Delta_{p,g}u+h(x)u^{p-1}=\widetilde{f}(x)u^{p^*-1}+\frac{\widetilde{a}(x)}{u^{p^*+1}},
\end{equation}
with
\begin{equation*}
\widetilde{f}=c^{p^*-p}f,\quad \widetilde{a}=\frac{a}{c^{p^*+p}}.
\end{equation*}
We wish to find a suitable constant $c>0$ such that our new coefficients $\widetilde{f}$ and $\widetilde{a}$ verify the conditions in Proposition \ref{prop4.1}
and \ref{prop4.2}. Clearly, if $u$ is a solution of Eq. \eqref{e4.44}, then $cu$ will solve Eq. \eqref{eq:M} accordingly. Obviously, the coefficient $h$
remains unchanged after scaling and we also have $\lambda_f=\lambda_{\widetilde{f}}$ since $c>0$. Therefore, the following conditions
\begin{equation*}
|h|<\lambda_{\widetilde{f}},\quad \widetilde{a}>0, \quad \int_M\widetilde{f}\,dv_g<0, \quad\sup_M{\widetilde{f}}^+>0
\end{equation*}
are fulfilled. Besides, it is easy to see that
\begin{equation*}
\frac{\sup_M\widetilde{f}}{\int_{M}|\widetilde{f}^-|\,dv_g}=\frac{\sup_Mf}{\int_{M}|f^-|\,dv_g}.
\end{equation*}
We now wish to remove \eqref{e4.1} but still keep other conditions. In other words, we have to choose a
suitable $c$ such that the following conditions
\begin{equation}\label{e4.46}
|h|\leq \frac{\eta_0}{p^*}\int_{M}|\widetilde{f}^-|dv_g,
\end{equation}
and
\begin{equation}\label{e4.47}
\sup_M\widetilde{f}<\mathcal{C}_2,
\end{equation}
and
\begin{equation}\label{e4.48}
\int_{M}\widetilde{a}\,dv_g< \frac{p}{2(n-p)}\Big(\frac{2n-p}{2(n-p)}\Big)^{\frac{2n}{p}-1}\Big(\frac{|h|}{\int_{M}|\widetilde{f}^-|dv_g}\Big)^{\frac{2n}{p}}\int_M|\widetilde{f}^-|\,dv_g
\end{equation}
hold. Indeed, \eqref{e4.46} and \eqref{e4.48} can be rewritten as the following
\begin{equation}\label{e4.49}
|h|\leq \frac{\eta_0c^{p^*-p}}{p^*}\int_{M}|\widetilde{f}^-|dv_g
\end{equation}
and
\begin{equation*}
\int_{M}{a}\,dv_g< \frac{p}{2(n-p)}\Big(\frac{2n-p}{2(n-p)}\Big)^{\frac{2n}{p}-1}\Big(\frac{|h|}{\int_{M}|{f}^-|dv_g}\Big)^{\frac{2n}{p}}\int_M|{f}^-|\,dv_g.
\end{equation*}
Thus, the condition \eqref{e1.3} is invariant under the variable change. In view of \eqref{e4.49}, we can choose
$$
c=\left(\frac{p^*|h|}{\eta_0\int_{M}|\widetilde{f}^-|dv_g}\right)^{\frac{1}{p^*-p}}.
$$
It suffices to prove that this particular choice of $c$ and the condition \eqref{e1.4} are enough to guarantee \eqref{e4.47}.
Notice that
$$
\sup_M\widetilde{f}=c^{p^*-p}\sup_Mf=\left(\frac{p^*|h|}{\eta_0\int_{M}|\widetilde{f}^-|dv_g}\right)\big(\sup_Mf\big)
=\frac{|h|p^*}{\eta_0}\frac{\sup_Mf}{\int_M|f^-|\,dv_g}.
$$
Therefore, if we assume
$$
\frac{\sup_Mf}{\int_M|f^-|\,dv_g}<\frac{\eta_0}{|h|p^*}\mathcal{C}_2,
$$
then the condition \eqref{e4.47} holds. In conclusion, if the constant $\mathcal{C}$ in the Theorem \ref{theorem1} equals
\begin{equation*}
\min\left\{\mathcal{C}_1, \frac{\eta_0}{|h|p^*}\mathcal{C}_2\right\}
\end{equation*}
we know that Eq. \eqref{eq:M} has at least two positive solutions. This finishes the proof of Theorem \ref{theorem1}.

\section{Proof of Theorem \ref{theorem2}}

In this section, we prove Theorem \ref{theorem2} which provides a sufficient condition for the solvability of \eqref{eq:M}.
As in the previous sections, we need to study the asymptotic behavior of $\mu_{k,q}^{\varepsilon}$ for small $k$ and large $k$, respectively.

We first assume that $f\leq 0$ but not strictly negative. We consider two possible cases.

\textbf{Case I.} $\sup_Mf=0$ and $\int_{\{f=0\}}1\,dv_g=0$. In this case, there holds $f<0$ almost everywhere in $M$ which
implies that $\mathcal{A}=\emptyset$. Hence, it holds $\lambda_f=+\infty$. However, for each $\eta\neq0$, $\lambda_{f,\eta,q}$ is well defined as in
\eqref{eq2.4}, and is monotone decreasing with respect to $\eta$ whose proof is exactly the same as the proof of Lemma \ref{lemma3.3}.
Moreover, we have the following lemma which is an analogous version of Lemma \ref{lemma3.5}, and we omit its proof.
\begin{lemma}\label{lemma5.1}
There exists $\eta_0$ such that for all $\eta<\eta_0$, there exists $q_{\eta}\in (p^{\flat}, p^*)$ such that
$\lambda_{f,\eta,q}>|h|$ for every $q\in (q_{\eta}, p^*)$.
\end{lemma}

\textbf{Case II.} $\sup_Mf=0$ and $\int_{\{f=0\}}1\,dv_g>0$. Under this case, $\lambda_f$ is well defined and finite.
A careful analysis shows that all results from section \ref{lfeq} remain hold. So we omit it here.

We are now in a position to study the behavior of $\mu_{k,q}^{\varepsilon}$ for $k\rightarrow +\infty$ when $\sup_Mf=0$.
\begin{proposition}\label{prop5.1}
Suppose $\sup_Mf=0$. If

\noindent- either $\int_{\{f=0\}}1\,dv_g=0$,

\noindent- or $\int_{\{f=0\}}1\,dv_g>0$ and $\lambda_f>|h|$,\\
then $\mu_{k,q}^{\varepsilon}\rightarrow +\infty$ as $k\rightarrow +\infty$ for any $\varepsilon\geq0$ sufficiently small
and any $q$ sufficiently close to $p^*$ but all are fixed.
\end{proposition}

\begin{proof}
We begin to prove that there exists some $\eta_0>0$ and its corresponding $q_{\eta_0}\in (p^{\flat}, p^*)$ such that
$\delta_0=(\lambda_{f, \eta_0, q}+h)/p>0$ for any $q\in (q_{\eta_0}, p^*)$.
As in the previous mentioned, we consider two cases separately.

\noindent \textbf{Case 1.} Suppose that $\sup_Mf=0$ and $\int_{\{f=0\}}1\,dv_g=0$.
In this case, $\lambda_f=+\infty$. Since $h$ is fixed, by Lemma \ref{lemma5.1},
there exists $\eta_0$ and its corresponding $q_{\eta}\in (p^{\flat}, p^*)$ such that $\lambda_{f,\eta,q}+h>0$ for any $q\in (q_{\eta}, p^*)$.
This proves the positivity of $\delta_0$.

\noindent \textbf{Case 2.} Suppose that $\sup_Mf=0$ and $\int_{\{f=0\}}1\,dv_g>0$.
In this case, $\lambda_f$ is well defined and finite. Notice that $\lambda_f>|h|$.
Since all results in section \ref{lfeq} still hold, as in the proof of Proposition \ref{prop3.2},
there exists $0<\eta_0<2$ and its corresponding $q_{\eta_0}\in (p^{\flat}, p^*)$ such that
$$0\leq\lambda_f-\lambda_{f, \eta_0, q}<\frac{1}{2}(\lambda_f-|h|)$$
for any $q\in (q_{\eta_0}, p^*)$. Therefore, $\delta_0>\frac{1}{2p}(\lambda_{f}+h)>0$.

Now having the strictly positivity of $\delta_0$ we can easily go through the proof of Proposition \ref{prop3.2} and get
$G_q(u)\geq mk^{\frac{p}{q}}$ where $m$ is given as in \eqref{e3.22}. This implies that
$\mathcal{I}_q^{\varepsilon}(u) \geq  mk^{\frac{p}{q}}$ since $\sup_Mf=0$. Since $\delta_0$ has a strictly positive lower bound, so does $m$.
The proof now follows easily.
\end{proof}
\begin{remark}
For small $k$, using the same argument as in Lemma \ref{asymptotic1} (i), one has $\mu_{k,q}^{k^{\frac{2}{q}}}\rightarrow +\infty$ as $k\rightarrow 0^+$.
Thus, compare with the case $\sup_Mf>0$, the curve $k\mapsto \mu_{k,q}^{\varepsilon}$ takes a shape as shown in Figure \ref{Figure2}.
\end{remark}
\begin{figure}
  \centering
  \includegraphics[width=7cm]{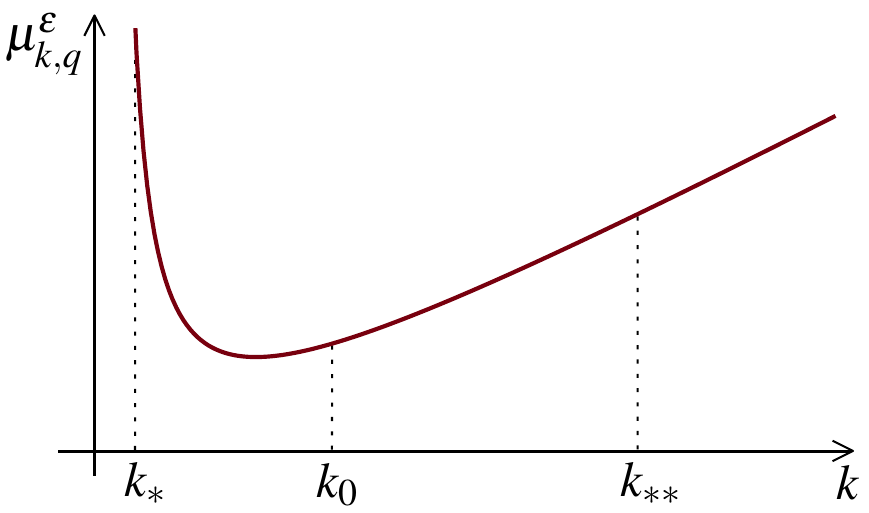}\\
  \caption{The asymptotic behavior of $\mu_{k,q}^{\varepsilon}$ when $\sup\limits_Mf\leq0$.}\label{Figure2}
\end{figure}
We are now in a position to prove Theorem \ref{theorem2} (1) which is similar to the proof of Theorem \ref{theorem1},
and  therefore we just sketch it and omit some details.

\begin{proof}[\textup{\textbf{Proof of Theorem \ref{theorem2} (1)}}]
 Suppose that $\sup_Mf=0$. Let $q\in (q_{\eta_0}, p^*)$. Since
\begin{equation*}\label{eq5.1}
\mathcal{I}_q^{\varepsilon}(k^{\frac{1}{q}})=\frac{h}{p}k^{\frac{p}{q}}-\frac{k}{q}\int_{M}f\,dv_g+\frac{1}{q}\int_{M}\frac{a}{(k^{\frac{2}{q}}
+\varepsilon)^{\frac{q}{2}}}\,dv_g,
\end{equation*}
by solving the following equation
\begin{equation*}\label{eq5.2}
\frac{h}{p}k^{\frac{p}{q}}-\frac{k}{q}\int_{M}f\,dv_g=0,
\end{equation*}
one can easily get
\begin{equation*}\label{eq5.3}
\mu_{k_0,q}^{\varepsilon}\leq \mathcal{I}_q^{\varepsilon}(k_0^{\frac{1}{q}})=\frac{1}{q}\int_{M}\frac{a}{(k_0^{\frac{2}{q}}
+\varepsilon)^{\frac{q}{2}}}\,dv_g<\frac{1}{pk_0}\int_{M}a\,dv_g,
\end{equation*}
where
\begin{equation*}\label{eq5.4}
k_0=\left(\frac{q}{p}\frac{h}{\int_Mf\,dv_g}\right)^{\frac{q}{q-p}}.
\end{equation*}
It is then easy to find $k_1$ and $k_2$ independent of both $q$ and $\varepsilon$ such that $k_1<k_0<k_2$ with $k_2>1$.
Now according to the asymptotic behavior of $\mu_{k,q}^{\varepsilon}$, one can find $k_{*}$ and $k_{**}$ independent of both $q$ and $\varepsilon$
with $k_*<k_1<k_0<k_2<k_{**}$ such that $\mu_{k_0,q}^{\varepsilon}<\min\{\mu_{k_*,q}^{\varepsilon}, \mu_{k_{**},q}^{\varepsilon}\}$.
Then we define
\begin{equation*}\label{e5.5}
\mathcal{E}_{1,q}^{\varepsilon}=\inf_{u\in \mathcal{D}_{k,q}}\mathcal{I}_q^{\varepsilon}(u)
\end{equation*}
for each $\varepsilon$ and $q$ fixed, where $\mathcal{D}_{k,q}=\{u\in H_1^p(M): \|u\|_q^q=k \,\,\text{and } k_*\leq k\leq k_{**}\}$.
At this point, one can use exactly the same arguments as in the proof of Proposition \ref{prop4.1} to obtain a positive solution of Eq. \eqref{eq:M}.
\end{proof}

Let us now assume that $\sup_Mf< 0$. It suffices to study the asymptotic behavior of $\mu_{k,q}^{\varepsilon}$ for large $k$.
Clearly, for any $u\in \mathcal{B}_{k,q}$, we have
\begin{equation*}\label{e5.6}
{\mathcal{I}}_q^{\varepsilon}(u)\geq\frac{h}{p}k^{\frac{p}{q}}+\frac{k}{p^*}\sup_{M}|f|.
\end{equation*}
We then immediately see that $\mu_{k,q}^{\varepsilon}\rightarrow +\infty$ as $k\rightarrow +\infty$ since ${p}/{q}<1$ (see Figure \ref{Figure2}).
With the same idea above, we can conclude that Eq.\eqref{eq:M} admits a positive solution. This is the content of the following result whose
proof is straightforward.
\begin{proposition}\label{prop5.3}
If $\sup_Mf<0$, then Eq.\eqref{eq:M} admits a positive solution.
\end{proposition}
\begin{remark}
One can also prove Theorem \ref{theorem2} (2) with the help of the classical sub- and supersolution method. See \cite{CIP2} for the case $p=2$ and \cite{Druet} for the case $a(x)\equiv 0$.
\end{remark}


\section*{Acknowledgments}
\addcontentsline{toc}{subsection}{Acknowledgements}
This work was supported by National Natural Science Foundation of China (Grants No. 11771342 and 11571259),
and was partially supported by the Natural Science Foundation of Hubei Province (Grants No. 2019CFA007).


\end{document}